\documentclass[12pt]{amsart}
\usepackage{amsmath,amssymb,latexsym, amsthm, amscd, mathrsfs, stmaryrd} 
\usepackage{todonotes}

\usepackage[all]{xy}
\usepackage{hyperref}
\usepackage{color}
\usepackage{verbatim}
\hypersetup{colorlinks,linkcolor=blue,urlcolor=cyan,citecolor=blue}

\newcommand{\nc}{\newcommand}
\nc{\browntext}[1]{\textcolor{brown}{#1}}
\nc{\greentext}[1]{\textcolor{green}{#1}}
\nc{\redtext}[1]{\textcolor{red}{#1}}
\nc{\bluetext}[1]{\textcolor{blue}{#1}}
\nc{\brown}[1]{\browntext{ #1}}
\nc{\green}[1]{\greentext{ #1}}
\nc{\red}[1]{\redtext{ #1}}
\nc{\blue}[1]{\bluetext{ #1}}
\nc{\zb}[1]{\redtext{From zb: #1}}
\def \bi{{\imath}}

\newcommand{\tK}{\widetilde{K}}
\newcommand{\tJ}{\widetilde{J}}

\newtheorem{thm}{Theorem}  [section]

\newtheorem{lem}[thm]{Lemma}
\newtheorem{prop}[thm]{Proposition}

\newtheorem{definition}[thm]{Definition}
\newtheorem{example}[thm]{Example}

\theoremstyle{remark}
\newtheorem{rem}{Remark}[section]

\numberwithin{equation}{section}

\newcommand{\mbf}{\mathbf}

\newcommand{\mrm}{\mathrm}
\newcommand{\A}{\mathcal A}
\newcommand{\ang}[1]{\langle #1 \rangle}

\newcommand{\bbinom}[2]{\begin{bmatrix}#1 \\ #2\end{bmatrix}}

\newcommand{\even}{{\bar 0}}
\newcommand{\odd}{{\bar 1}}
\newcommand{\ev}{\mathrm{ev}}
\newcommand{\od}{\mathrm{odd}}
\newcommand{\dvev}[1]{\tfk_{\even}^{{(#1)}}}
\newcommand{\dvo}[1]{\tfk_{\odd}^{{(#1)}}}

\newcommand{\diag}{\mrm{diag}}

\newcommand{\kk}{h}
\newcommand{\la}{\lambda}
\newcommand{\LR}[2]{\left\llbracket \begin{matrix} #1\\#2 \end{matrix} \right\rrbracket}
\newcommand{\F}{\mathbb F}
\newcommand{\N}{\mathbb N}
\newcommand{\K}{\mathbb K}
\newcommand{\id}{\text{id}}

\newcommand{\wt}{\text{wt}}
\newcommand{\wtd}{\widetilde}

\newcommand{\one}{\mathbf 1}
\newcommand{\oldone}{\mathbf 1}
\newcommand{\onestar}{\one^\star}
\newcommand{\ov}{\overline}
\newcommand{\qbinom}[2]{\begin{bmatrix} #1\\#2 \end{bmatrix} }

\newcommand{\osp}{\mathfrak{\lowercase{osp}}}
\newcommand{\tfk}{B}
\newcommand{\UU}{\mbf U}
\newcommand{\UUdot}{\dot{\mbf U}}

\newcommand{\UA}{{}_\A{\mbf U}}
\newcommand{\UAdot}{{}_\A{\dot{\mbf U}}}
\newcommand{\UUi}{{\mbf U}^\imath}

\newcommand{\vs}{\varsigma}
\newcommand{\Y}{\check{E}}
\newcommand{\Z}{\mathbb Z}
\newcommand{\I}{\mathbb I}

\def \m{{m}}

\def \U{{\mathbf U}}

\begin{document}
	
	\title[A Serre presentation for the $\imath${}quantum covering groups]{A Serre presentation for the $\imath${}quantum covering groups}
	
	\author[Christopher Chung]{Christopher Chung}
	\address{Department of Mathematics, University of Virginia, Charlottesville, VA 22904}
	\email{cc2wn@virginia.edu}

	
	\keywords{Quantum groups, quantum covering groups, quantum symmetric pairs, Serre relations, bar involution}
	
	\begin{abstract}
		Let $ (\UU, \UUi) $ be a quasi-split quantum symmetric pair of Kac-Moody type. The $\imath$quantum group $\UUi$ admits a Serre presentation featuring the $\imath$-Serre relations in terms of $\imath$-divided powers. Generalizing this result, we give a Serre presentation $ \UUi_\pi $ of quantum symmetric pairs $ (\UU_\pi, \UUi_\pi) $ for quantum covering algebras $\UU_\pi$, which have an additional parameter $ \pi $ that specializes to the Lusztig quantum group when $ \pi = 1 $ and quantum supergroups of anisotropic type when $ \pi = -1 $. We give a Serre presentation for $ \UUi_\pi $, introducing the {\em $\imath^\pi$-Serre relations} and {\em $\imath^\pi$-divided powers}.
	\end{abstract}
	\maketitle

	\setcounter{tocdepth}{1}
	\tableofcontents

	\section{Introduction}

	\subsection{}
	
	The Drinfeld-Jimbo quantum group $\U$ is a $q$-deformation of the universal enveloping algebra of a symmetrizable Kac-Moody algebra, with Chevalley generators $E_i, F_i, K_i^{\pm 1}$, for $i\in I$. $ \UU $ admits a familiar presentation, its {\em Serre presentation}, with a key feature being the $q$-Serre relations among the $E_i$'s and $F_i$'s. In terms of divided powers $F_i^{(n)} = F_i^n/[n]_{q_i}^!$ (cf. \cite{Lu93} where $[n]_{q_i}^!$ are the so called quantum factorials), the $q$-Serre relations among the $F_i$'s has a compact form: for $i \neq j \in I$,
	\begin{align}
	\label{eq:qSerre}
	\sum_{n=0}^{1-a_{ij}} (-1)^n   F_i^{(n)}F_j F_i^{(1-a_{ij}-n)}=0.
	\end{align}
	The quantum group $\U$ is a Hopf algebra with a comultiplication $\Delta : \UU \to \UU \otimes \UU$.
	
	Quantum symmetric pairs $(\U, \UUi)$, are deformations of classical symmetric pairs which are defined using Satake diagrams, which are Dynkin diagrams with some nodes blackened and other nodes connected in pairs by a diagram involution. The theory of quantum symmetric pairs was systematically studied by Letzter for $\U$ of finite type  (cf. \cite{Le99, Le02}) and in Kac-Moody type the theory was further developed by Kolb \cite{Ko14}. 
	The {\em $\imath${}quantum group} $\UUi$ is a (right) coideal subalgebra of $ \UU$: it satisfies the property that $\Delta: \UUi \rightarrow \UUi \otimes \U$. Main generators of $\UUi$ are defined in terms of generators of $\UU$ using an embedding formula cf. \eqref{eq:def:ff}:
	\begin{equation}
	\label{eq:emb}
	B_i =F_{i} + \vs_i E_{\tau i} \tK^{-1}_i 
	, \quad \text{ for }i\in I,
	\end{equation}
	where $
	\vs =(\vs_i)_{i\in I}, 
	$ are parameters. 
	
	Borrowing terminologies from real Lie groups, we will call a quantum symmetric pair and an $\imath${}quantum group {\em quasi-split} (and respectively, {\em split}) if the underlying Satake diagram contains no black node (respectively, with the trivial involution in the Satake diagram). These can be thought of as the $\imath${}quantum groups associated to the Chevalley involution $\omega$, coupled with a diagram involution $\tau$ (which is allowed to be the identity). A quasi-split $\imath${}quantum group takes the generalized Cartan matrix and a diagram involution $\tau$ as its only inputs.
	%
	
	In \cite{CLW18}, a Serre presentation uniform relations for the quasi-split $\imath${}quantum groups of Kac-Moody type with general parameters is formulated precisely, generalizing the work of Letzter in finite type and Kolb in Kac-Moody type for $ |a_{ij}| \le 3 $, cf.  \cite{Le02, Le03, Ko14}. A centerpiece of the Serre presentation for $ \UUi $ is the {\em $\imath$-Serre relations} between $B_i$ and $B_j$ for $\tau i = i\neq j$. These relations can be expressed in striking similarity to the $q$-Serre relation \eqref{eq:qSerre}: for any fixed $\overline{p}\in \Z_2 =\{\bar 0, \bar 1\}$,
	\begin{equation}
	\label{eq:iSe}
	\sum_{n=0}^{1-a_{ij}} (-1)^n  B_{i,\overline{a_{ij}}+\overline{p}}^{(n)}B_j B_{i,\overline{p}}^{(1-a_{ij}-n)}  =0,
	\end{equation}
	where the {\em $\imath$-divided powers} $B^{(m)}_{i,\ov{p}}$ are polynomials (compare Lusztig's divided powers, which are monomials) in $B_i$ which depend on a parity $\ov p$ arising from the parities of the highest weights of highest weight $\U$-modules when evaluated at the coroot $h_i$. The $\imath$-divided powers were introduced in \cite{BW13, BeW18}, and are canonical basis elements 
	for (the modified form of) $\UUi$ in the sense of \cite{BW18b}. Writing the $\imath$-Serre relations \eqref{eq:iSe} in terms of $\imath$-divided powers provided a uniform reformulation of complicated case-by-case relations for the cases $|a_{ij}|\le 3$ in \cite{Ko14, BK19}, which enabled the method of proof in \cite[\S4]{CLW18}.
	
	%
	
	A precise formulation of the Serre presentation is crucial to the formulation of a bar involution on a general $\imath${}quantum group $\UUi$ as predicted in \cite{BW13}; it allows one to write down the constraints that the parameters should satisfy \cite{BK15}. The bar involution on $\UUi$ is a basic ingredient for the canonical basis  for $\UUi$ \cite{BW18b, BW18c}. The $\imath${}divided powers are also a key component in constructing the Frobenius-Lusztig homomorphism for $\imath${}quantum groups at roots of unity in \cite{BaS19}.
	
	%
	
	%
	
	\subsection{}

	A quantum covering group $\UU_\pi$, introduced in \cite{CHW13} (cf. \cite{HW15}) is an algebra defined via a super Cartan datum $I$ (a finite indexing set associated to Kac-Moody superalgebras with no isotropic odd roots). $ \UU_\pi $ depends on two parameters $ q $ and $ \pi $, where $ \pi^2 = 1 $. A quantum covering group specializes at $ \pi = 1 $ to the quantum group above, and at $ \pi = -1 $ to a quantum supergroup of anisotropic type (see \cite{BKM98}). In addition to the usual Chevalley generators, we have generators $J_i$ for each $ i \in I $. If one writes $ K_i $ as $ q^{h_i} $, then analogously we will have $ J_i = \pi^{h_i} $.  The parameter $\pi$ can be thought of as a shadow of the parity shift functor $\Pi$ in Hill and Wang's (\cite{HW15}) categorification of quantum groups by the \emph{spin} quiver Hecke superalgebras introduced in \cite{KKT16}. Since then, further progress has been made on the odd/spin/super categorification of 
	quantum covering groups; see \cite{KKO14, EL16, BE17}. 

	\subsection{}
	In this paper, we formulate quasi-split quantum symmetric pairs $(\UU_\pi,\UUi_\pi)$ for quantum covering groups. The $\imath${}quantum covering group $ \UUi_\pi $ is by definition a subalgebra of $ \UU_\pi $ that satisfies the coideal property, with the same embedding formulas as \ref{eq:emb} cf. \eqref{eq:def:ff}. 
	
	A first step in generalizing the Serre presentation in \cite{CLW18} involved defining suitable $\pi$-analogues of the $\imath$-divided powers, which we call the {\em $\imath^\pi$-divided powers} with $ \pi_i $ and $ J_i $ incorporated judiciously. The $\imath$-divided powers satisfy explicit closed form formulas in terms of the PBW basis for $ \UU $, which were crucial for the proof of the $\imath$-Serre relation \eqref{eq:iSe}. We were able to deduce that the $\imath^\pi$-divided powers also satisfy similar expansion formulas. With this, we were able to prove the validity of the  {\em $\imath^\pi$-Serre relations} \eqref{eq:piSerre} below by a parallel strategy of reducing to a $ (q,\pi) $-binomial identity, \cite[\S4]{CLW18}. The $\imath^\pi$-divided powers, accompanying expansion formulas and $\imath^\pi$-Serre relations specialize to those contained in \cite{BeW18} when we set $ \pi = 1 $.
	
	As a notational convenience, in the rest of this paper we will drop the subscript $ \pi $, so $ \UU $ is understood to refer to quantum covering group. We will explicitly mention the quantum groups when we specialize $ \pi = 1 $. 

	\subsection{}
	We will indicate here a few applications: in an upcoming paper \cite{C19b}, the results in this paper will be used to construct a quasi $K$-matrix and prove the integrality of its action cf. \cite{BW18b,BW18c}. This will enable us to define based modules for the $\imath${}quantum covering groups, and develop canonical basis for these modules. The $ \imath^\pi$-divided powers are also expected to play a role in defining a version of a Frobenius-Lusztig homomorphism for quantum super symmetric pairs at roots of unity.

	\subsection*{Acknowledgments.}
	The author expresses deep gratitude to his advisor Weiqiang Wang for patience and guidance. This research is partially supported by Wang's NSF grant DMS-1702254, including GRA supports for the author. He also thanks Arun Kannan for his help in verifying the expansion formulas in Theorem~\ref{thm:iDP:ev} and Theorem~\ref{thm:iDP:odd} using the programming language Singular, based on C. Berman's Mathematica code.
	
	\section{The preliminaries}
	\label{section:preliminaries}

	\subsection{Quantum covering groups}
	\label{subsec:QCG}
	
	Here, we recall the definition of a quantum covering group from \cite{CHW13} starting with a {\em super Cartan datum} and a root datum. A {\em Cartan datum} is a pair $(I,\cdot)$ consisting of a finite
	set $I$ and a symmetric bilinear form $\nu,\nu'\mapsto \nu\cdot\nu'$
	on the free abelian group $\Z[I]$ with values in $\Z$ satisfying
	\begin{enumerate}
		\item[(a)] $d_i=\frac{i\cdot i}{2}\in \Z_{>0}$;
		
		\item[(b)]
		$2\frac{i\cdot j}{i\cdot i}\in -\N$ for $i\neq j$ in $I$, where $\N
		=\{0,1,2,\ldots\}$.
	\end{enumerate}
	If the datum can be decomposed as $ I=I_{\bf \even} \coprod I_{\bf \odd} $ such that
	\begin{enumerate}
		\item[(c)] $I_{\bf \odd}\neq\emptyset$,
		\item[(d)] $2\frac{i\cdot j}{i\cdot i} \in 2\Z$ if $i\in I_{\bf \odd}$,
		\item[(e)] $d_i\equiv p(i) \mod 2, \quad \forall i\in I.$
	\end{enumerate}
	then we will called it a (bar-consistent) {\em super Cartan datum}. Condition [(e)] is known as the `bar-consistency' condition and is almost always satisfied for super Cartan data of finite or affine type (with one exception).
	
	Note that (d) and (e) imply that
	\begin{enumerate}
		\item[(f)] $i\cdot j\in 2\Z$ for all $i,j\in I$.
	\end{enumerate}
	
	The $i\in I_{\bf \even}$ are called even, $i\in I_{\bf \odd}$ are called odd. We
	define a parity function $p:I\rightarrow\{0,1\}$ so that $i\in
	I_{\bar{p(i)}}$. We extend this function to the homomorphism
	$p:\Z[I]\rightarrow \Z$. Then $p$ induces a $\Z_2$-grading on
	$\Z[I]$ which we shall call the parity grading.
	
	A super Cartan datum $(I,\cdot)$ is said to be of {\em finite}
	(resp. {\em affine}) type exactly when $(I,\cdot)$ is of
	finite (resp. affine) type as a Cartan datum (cf. \cite[\S 2.1.3]{Lu93}).
	In particular, the only super Cartan datum of finite type is the one
	corresponding to the Lie superalgebras of type
	$B(0,n)$ for $n\geq 1$ i.e. the orthosymplectic Lie superalgebras $ \osp(1|2n) $.
	
	A {\em root datum} associated to a super Cartan datum $(I,\cdot)$
	consists of
	\begin{enumerate}
		\item[(a)]
		two finitely generated free abelian groups $Y$, $X$ and a
		perfect bilinear pairing $\ang{\cdot, \cdot}:Y\times X\rightarrow \Z$;
		
		\item[(b)]
		an embedding $I\subset X$ ($i\mapsto i'$) and an embedding $I\subset
		Y$ ($i\mapsto i$) satisfying
		
		\item[(c)] $\ang{i,j'}=\frac{2 i\cdot j}{i\cdot i}$ for all $i,j\in I$.
	\end{enumerate}
	We will always assume that the root datum is {\em $X$-regular} (respectively {\em $Y$-regular}) image of the embedding $I\subset X$
	(respectively, the image of the embedding $I\subset Y$) is linearly
	independent in $X$ (respectively, in $Y$).
	
	The matrix $ A := (a_{ij}) := \ang{i,j'} $ is a {\em symmetrizable generalized super Cartan matrix}: if $ D = \diag(d_i\,|\, i \in I) $, then $ DA $ is symmetric.
	
	
	Let $\pi$ be a parameter such that
	$$\pi^2=1.
	$$
	For any $i\in I$, we set
	$$q_i=q^{i\cdot i/2}, \qquad \pi_i=\pi^{p(i)}.
	$$
	Note that when the datum is consistent, $\pi_i=\pi^{\frac{i\cdot
			i}{2}}$; by induction, we therefore have
	$\pi^{p(\nu)}=\pi^{\nu \cdot \nu/2}$ for $\nu\in \Z[I]$. 
	We extend this notation so that if
	$\nu=\sum \nu_i i\in \Z[I]$, then
	$$
	q_\nu=\prod_i q_i^{\nu_i}, \qquad \pi_\nu=\prod_i \pi_i^{\nu_i}.
	$$
	For any ring $R$ we define a new ring $R^\pi =R[\pi]/(\pi^2-1)$ 
	(with $\pi$ commuting with $R$). Below, we will work over $\K(q)^\pi$ where $\K$ is a field of characteristic $0$.
	
	Recall also the {\em $(q,\pi)$-integers} and {\em $(q,\pi)$-binomial coefficients} in \cite{CHW13}: we shall denote
	\[
	[n]=\bbinom{n}{1}=\frac{(\blue{\pi} q)^n-q^{-n}}{\blue{\pi} q - q^{-1}}\quad
	\text{for } n\in\Z,
	\]
	\[
	[n]^!=\prod_{s=1}^n [s]\quad \text{for } n\in\N,
	\]
	and with this notation we have
	\[
	\bbinom{m}{n}=\frac{[m]^!}{[n]^![m-n]^!}\quad \text{for
	}0\leq n \leq m.
	\]
	We denote by $[n]_{i}, [m]_{i}^!,$ and $\qbinom{n}{m}_{i}$ the variants of $[n], [m]!,$ and $\qbinom{n}{m}$ with $q$ replaced by $q_i$ and $ \pi $ replaced by $ \pi_i $, and $ \bbinom{m}{n}_{q^2} $ the variant with $q$ replacing $ q^2 $. 
	
	For any $i\neq j$ in $I$, we define the following polynomial in two (noncommutative) variables $x$ and $y$:
	\begin{equation}
	\label{qpipoly}
	F_{ij}(x,y) = \sum_{n=0}^{1-a_{ij}}(-1)^{n}\pi_i^{np(j)+\binom{n}{2}} \bbinom{1 - a_{ij}}{n}_i x^n y x^{1 - a_{ij} - n}.
	\end{equation}
	Also, we have 
	
	Assume that a root datum $(Y,X, \ang{\,,\,})$
	of type $(I,\cdot)$ is given. The {\em quantum covering group} $ \UU $ of
	type $(I,\cdot)$ is the associative $\K(q)^\pi$-superalgebra with generators
	\[\
	E_i\quad(i\in I),\quad F_i\quad (i\in I), \quad J_{\mu}\quad (\mu\in
	Y),\quad K_\mu\quad(\mu\in Y),
	\]
	with parity $p(E_i)=p(F_i)=p(i)$ and
	$p(K_\mu)=p(J_\mu)=0$, subject to the relations (a)-(f) below for
	all $i, j \in I, \mu, \mu'\in Y$:
	\[
	\tag{R1} K_0=1,\quad K_\mu K_{\mu'}=K_{\mu+\mu'},
	\]
	\[
	\tag{R2} J_{2\mu}=1, \quad J_\mu J_{\mu'}=J_{\mu+\mu'},
	\]
	\[
	\tag{R3} J_\mu K_{\mu'}=K_{\mu'}J_{\mu},
	\]
	\[
	\tag{R4} K_\mu E_i=q^{\ang{\mu,i'}}E_iK_{\mu}, \quad
	J_{\mu}E_i=\pi^{\ang{\mu,i'}} E_iJ_{\mu},
	\]
	\[
	\tag{R5} \; K_\mu F_i=q^{-\ang{\mu,i'}}F_iK_{\mu}, \quad J_{\mu}F_i=
	\pi^{-\ang{\mu,i'}} F_iJ_{\mu},
	\]
	\[
	\tag{R6} E_iF_j-\pi^{p(i)p(j)}
	F_jE_i=\delta_{i,j}\frac{\wtd{J}_{i}\wtd{K}_i-\wtd{K}_{-i}}{\pi_iq_i-q_i^{-1}},
	\label{R6}
	\]
	\[
	\tag{R7}(q,\pi)\text{-Serre relations}\qquad F_{ij}(E_i,E_j) = 0 = F_{ij}(F_i,F_j), \text{ for all }i \neq j.
	\label{R7}
	\]
	where for any element $\nu=\sum_i \nu_i i\in \Z[I]$ we have set
	$\wtd{K}_\nu=\prod_i K_{d_i\nu_i i}$, $\wtd{J}_\nu=\prod_i
	J_{d_i\nu_i i}$. In particular, $\wtd{K}_i=K_{d_i i}$,
	$\wtd{J}_i=J_{d_i i}$. Under the bar-consistency condition,
	$\wtd{J}_i=1$ for $i\in I_{\bf \even}$ while $\wtd{J}_i=J_i$ for $i\in
	I_{\bf \odd}$. Note that by the same condition $ a_{ij} $ is always even for $ i \in I_{\bf \odd} $, and so $ J_i $ is central for all $i \in I$. As usual, denote by $ \UU^- $, $ \UU^+ $ and $ \UU^0 $ the subalgebras of $ \UU $ generated by $ \{ E_i \,|\, i \in I \} $, $ \{ F_i \,|\, i \in I \} $ and $ \{ J_\mu, K_\mu \,|\, \mu \in Y \} $ respectively. Also denote $ \UU^{0'} = \{ J_i, K_i \,|\, i \in I \} $.
	
	If we write $ F_i^{(n)} = F_i^n/[n]^!_{i} $ and $ E_i^{(n)} = E_i^n/[n]^!_{i} $ for $ n \geq 1 $ and $ i \geq 1 $, then the $(q,\pi)$-Serre relations \eqref{R7} can be rewritten as:
	\begin{equation}
	\label{qpiFSerre}
	\sum_{n=0}^{1-a_{ij}}(-1)^{n}\pi_i^{np(j)+\binom{n}{2}}  F_i^{(n)} F_j F_i^{(1 - a_{ij} - n)} = 0
	\end{equation}
	and
	\begin{equation}
	\label{qpiESerre}
	\sum_{n=0}^{1-a_{ij}}(-1)^{n}\pi_i^{np(j)+\binom{n}{2}}  E_i^{(n)} E_j E_i^{(1 - a_{ij} - n)} = 0.
	\end{equation}
	
	%
	
	The following lemma is an analogue of \cite[Lemma~2.1]{CLW18}.
	\begin{lem}   \label{lemma:involution of U}
		There exists an involution $\varpi$ on the $\K(q)$-algebra $\U$ which sends
		\begin{align}
		\varpi: E_i\mapsto q_i^{-1}F_i\tK_i,\quad F_i\mapsto q_i^{-1} E_i\tK_i^{-1},\quad J_\mu\mapsto J_\mu,\quad K_\mu\mapsto K_\mu,\quad q\mapsto q^{-1}.
		\end{align}
		for any $i\in I$, $\mu\in Y$.
	\end{lem}
	
	\begin{proof}
		The verification that $ \varpi $ preserves the defining relations is straightforward; for instance
		\begin{align*}
		&(q_i^{-1} F_i \tK_i)(q_i^{-1} E_i \tK_i^{-1}) - \pi_i 
		(q_i^{-1} E_i \tK_i^{-1})(q_i^{-1} F_i \tK_i) \\
		&= q_i^{-2} F_i (\tK_i E_i \tK_i^{-1}) - q_i^{-2} E_i (\tK_i^{-1} F_i \tK_i) \\
		&= q_i^{-2} (q_i^2 F_i E_i - \pi_i q_i^2 E_i F_i) \\
		&= -\pi_i \frac{\tJ_i \tK_i - \tK_{-i}}{\pi_i q_i - q_i^{-1}} = \frac{\tJ_i \tK_i - \tK_{-i}}{\pi_i q_i^{-1} - q_i},
		\end{align*}
		and so $ \varpi $ preserves relation \eqref{R6}.
	\end{proof}
	
	\subsection{The algebra $\UUdot$}
	\label{subsec:UUdot}
	
	Recall \cite{Lu93,Cl14} that the modified form of $\UU$, denoted by $\UUdot$, is a (non-unital) $\K(q)^\pi$-algebra generated by $\oldone_\la, E_i \oldone_\la$, $F_i \oldone_\la$, for $i\in I, \la \in X$, where $\oldone_\la$ are orthogonal idempotents. Let
	$
	\A=\Z^\pi[q,q^{-1}].
	$
	There is an $\A$-subalgebra  $_\A \UUdot$ generated by $E_i^{(n)}\oldone_\la, F_i^{(n)}\oldone_\la$ for $i\in I$ and $n\geq 0$ and $\la \in X$.
	Note that $\UUdot$ is naturally a $\UU$-bimodule, and in particular we have
	\begin{align*}
	K_h\oldone_\la=\oldone_\la K_h &=q^{\langle h,\la\rangle}\oldone_\la, \; \forall h\in Y.
	\end{align*}
	
	We have the mod $2$ homomorphism $\Z \rightarrow \Z_2, k \mapsto \ov k$, where $\Z_2=\{\bar{0},\bar{1}\}$. Let us fix an $i\in I$.  Define
	\begin{equation}
	\UUdot_{i,\ev}:=\bigoplus_{\la:\, \langle h_i,\la\rangle \in 2\Z}\,\UUdot \oldone_{\la},
	\qquad
	\UUdot_{i,\od}:=\bigoplus_{\la:\, \langle h_i,\la\rangle \in1+2\Z}\, \UUdot\oldone_{\la}.
	\end{equation}
	Then $\UUdot= \,\UUdot_{i,\ev}\oplus{\UUdot_{i,\od}}$.
	Similarly, letting $_\A\UUdot_{i,\ev} =\UUdot_{i,\ev} \cap _\A\UUdot$ and $_\A\UUdot_{i,\od} =\UUdot_{i,\od} \cap _\A\UUdot$, we have $_\A\UUdot = {}_\A\UUdot_{i,\ev} \oplus {}_\A\UUdot_{i,\od}$.
	
	For our later use, with $i\in \I$ fixed once for all, we need to keep track of the precise value $\langle h_i,\la\rangle$ in an idempotent $\oldone_\la$ but do not need to know which specific weights $\la$ are used. Thus it is convenient to introduce the following generic notation
	\begin{align}
	\label{eq:1star}
	\onestar_m =\onestar_{i,m}, \qquad \text{ for }m\in \Z,
	\end{align}
	to denote an idempotent $\oldone_\la$ for some $\la\in X$ such that $\m=\langle h_i,\la\rangle$.
	In this notation, the identities in \cite{Cl14} (with a correction provided in \cite[Lemma~3.2]{CSW18}) can be written as follows: for any $\m\in\Z$, $a,b\in\Z_{\geq0}$, and $i\neq j\in I$,
	\begin{align}
	E_i^{(a)}\onestar_{i,\m} &=\onestar_{i,\m+2a} E_i^{(a)},\quad F_i^{(a)}\onestar_{i,\m}=\onestar_{i,\m-2a} F_i^{(a)};
	\label{eqn: idempotent Ei Fi}\\
	E_j\onestar_{i,\m} &=\onestar_{i,\m+a_{ij}}E_j,  \qquad F_j\onestar_{i,\m}=\onestar_{i,\m-a_{ij}} F_j;
	\label{eqn: idempotent Ej Fj}\\
	F_{i}^{(a)} E_i^{(b)}\onestar_{i,\m} &= \sum_{j=0}^{\min\{a,b\}} \pi_i^{ab+jm+{j \choose 2}} \qbinom{a-b-\m}{j}_{i} E_i^{(b-j)} F_i^{(a-j)}\onestar_{i,\m};
	\label{eqn:commutate-idempotent3}\\
	E_i^{(a)} F_i^{(b)}\onestar_{i,\m} &=\sum_{j=0}^{\min\{a,b\}} \pi_i^{ab+{j+1 \choose 2}} \qbinom{a-b+\m}{j}_{i} F_i^{(b-j)} E_i^{(a-j)}\onestar_{i,\m}
	\label{eqn:commutate-idempotent4}.
	\end{align}
	From now on, we shall always drop the index $i$ to write the idempotents as $\onestar_m$.
	
	\begin{rem}
		\label{rem:u=0}
		If $u\in\U$ satisfies $u\onestar_{2k-1}=0$ for all possible idempotents $\onestar_{2k-1}$ with $k\in\Z$ (or respectively,  $u\onestar_{2k}=0$ for all possible $\onestar_{2}$ with $k\in\Z$), then $u=0$.
	\end{rem}

	\subsection{The $\imath${}quantum covering group $\UUi$}
	\label{subsec:irootdatum}
	Let $(Y,X, \langle\cdot,\cdot\rangle, \cdots)$ be a root datum of (super) type $(I, \cdot)$.  We call a permutation $\tau$ of the set $I$ an {\em involution} of the Cartan datum $(I, \cdot)$ if $\tau^2 =\id$ and $\tau i \cdot \tau j = i \cdot j$ for $i$, $j \in I$.  Note we allow $\tau =\id$.
	We will always assume that $\tau$ extends to an involution on $X$ and an involution on $Y$ (also denoted by $\tau$), respectively, such that the perfect bilinear pairing is invariant under the involution $\tau$.
	The permutation $\tau$ of $I$ induces an $\K(q)$-algebra automorphism of $\U$, 
	defined by
	\begin{align}
	\label{eq:tau}
	\tau:E_i \mapsto E_{\tau i}, \quad F_i \mapsto F_{\tau i}, \quad K_h \mapsto K_{\tau h}, \qquad \forall i\in I, h\in Y.
	\end{align}
	
	Define 
	\begin{align}
	\label{XY}
	\begin{split}
	Y^{\imath} &= \{h \in Y  \mid \tau(h)=-h \}.
	\end{split}
	\end{align}
	
	Just as in \cite{CLW18}, in this paper we will only consider the quasi-split case (corresponding to Satake diagrams without black nodes).
	
	\begin{definition}
		\label{def:UUi}
		The {\em quasi-split $\imath${}quantum group}, denoted by $\UUi_{\vs}$ or $\UUi$, is the $\K(q)$-subalgebra of $\U$ generated by
		\begin{align}
		B_i :=F_{i}  &+ \vs_i E_{\tau i} \tK^{-1}_i, \qquad \tJ_i \, \,(i \in I),
		\qquad K_{\mu}\, \,(\mu \in Y^{\imath}).
		\label{eq:def:ff}
		\end{align}
		Here the parameters
		\begin{equation}
		\label{parameters}
		\vs=(\vs_i)_{i\in I}\in (\K(q)^\times)^I,\qquad
		\end{equation}
		are assumed to satisfy Conditions \eqref{bar1}--\eqref{bar3} below:
		\begin{align}
		\ov{\vs_iq_i} &= \vs_iq_i \text{ if } \tau i=i \text{ and } a_{ij}\neq 0 \text{ for some } j\in I\setminus\{i\}; \label{bar1} \\
		\vs_{i} & =\vs_{{\tau i}} \text{ if }    a_{i,\tau i} =0;
		\label{bar2} \\
		\vs_{{\tau i}} &= \pi_i q_i^{-a_{i,\tau i}} \overline{\vs_i} \text{ if }    a_{i,\tau i} \neq 0.   \label{bar3}
		\end{align}
	\end{definition}
	
	The conditions on the parameters ensure that $\UUi$ admits a suitable bar-involution (see \S\ref{subsec:bar}).
	
	$\triangleright$ $\UUi$ is a (right) coideal subalgebra of $\U$, i.e., $\Delta: \UUi \longrightarrow \UUi \otimes \U$.
	
	$\triangleright$ In \cite{Ko14} and \cite{CLW18} an additional set of parameters $ \kappa_i $ is considered; in the setting of quantum covering groups the only interesting case ($\kappa_i \neq 0 $ for some $i \in \I$) exists in rank $2$ ($\osp(1|4)$), so we will omit this from general consideration.
	
	
	%
	
	\subsection{Structure and size of $ \UUi $}
	\label{UUisize}
	A few of the results on the size and structure of $ \UUi $ are collected here cf. \cite[\S5--6]{Ko14}. First, we define the projections $ P_\lambda $ and $ \pi_{\alpha,\beta} $ similarly to \cite[\S5.2]{Ko14}: by the triangular decomposition \cite[Corollary~2.3.3]{CHW13}, 
	$$ 
	\UU = \bigoplus_{\lambda \in Y} \UU^+ \UU_J K_\lambda S(\UU^-), 
	$$
	where $ \UU_J = \left< J_\mu \,|\, \mu \in Y \right> $ and $ S $ denotes the antipode of $ \UU $. For any $ \lambda \in Y $ let
	\begin{equation}
	\label{projection_P}
	P_\lambda : \UU \to \UU^+ \UU_J K_\lambda S(\UU^-)
	\end{equation}
	denote the projection with respect to this decomposition.
	
	Similarly, let 
	\begin{equation}
	\label{projection_pi}
	\pi_{\alpha,\beta} : \UU \to \UU_\alpha^+ \UU^0 \UU_{-\beta}^-
	\end{equation} 
	denote the projection with respect to the decomposition
	$$
	\UU = \bigoplus_{\alpha,\beta \in Y^+} \UU_\alpha^+ \UU^0 \UU_{-\beta}^-.
	$$   
	
	Because the embedding formulas for the $\imath${}quantum covering groups follow the same form as in \cite[(5.1)]{Ko14} (with $ X = \emptyset $ and $ s_i = 0 $), we have the following technical lemma, proved in the same way as in \emph{loc. cit.}:
	
	\begin{lem}
		Let $ \alpha,\beta \in Q^+ $. If $ \pi_{\alpha,\beta} (F_{ij}(B_i,B_j)) \neq 0 $ then $ \lambda_{ij} -\alpha \in Q^\Theta $ and $ \lambda_{ij} -\beta \in Q^\Theta $.
	\end{lem}

	Using this, we also have the following results about $ \UUi $: 
	
	\begin{prop}
		In $\UU$, we have $ P_{\lambda_{ij}}(F_{ij}(B_i,B_j)) = 0 $ for all $ i,j \in I $.
	\end{prop}
	
	\begin{prop}
		\label{Ko517}
		In $ \UUi $, we have the relation
		\begin{equation}
			F_{ij}(B_i,B_j) \in \sum_{\{J \in \mathcal{J}\,|\, \wt(J) < \lambda_{ij}\}} \UU^{0'}_\Theta B_J \text{ for all } i,j \in I.
		\end{equation}
	\end{prop}
	
	We now show that $ \UUi $ has the same size as $ \UU^- $, cf \cite[\S 6.1--2]{Ko14}.	For any multi-index $ J = (j_1,\ldots,j_n) $, define $ \wt(J) = \sum_{i = 1}^n \alpha_j $, and $ F_J = F_{j_1} \ldots F_{j_n} $ and $ B_J = B_{j_1} \ldots B_{j_n} $, and define $ |J| = n $. Let $ \mathcal{J} $ be a fixed subset of $ \bigcup_{n \in \N_0} I^n $ such that $ \{ F_J \,|\, J \in \mathcal{J} \} $ is a basis of $ \UU^- $, and hence a basis of $ \UU' $ as a left $ \UU^+ \UU^{0'} $-module. Define a filtration $ \mathcal{F}^* $ of $ \UU^- $ by $ \mathcal{F}^n(\UU^-) = \text{span}\{ F_J \,|\, J \in I^m, m \leq n \} $ for all $ n \in \N_0 $. By the homogeneity of the $ (q,\pi) $-Serre relations \eqref{qpiFSerre}, the set $ \text{span}\{ F_J \,|\, J \in \mathcal{J}, |J| = n \} $ forms a basis of $ \mathcal{F}^n(\UU^-) $. Then, we have the following proposition, cf. \cite[Prop~6.2]{Ko14}:
	
	\begin{prop}
		The set $ \{ B_J \,|\, J \in \mathcal{J} \} $ is a basis of the left (or right) $ \UU^+ \UU^{0'} $-module $ \UU^\imath $.
	\end{prop}
	
	\begin{proof}
		The argument is the same as the one in \cite[Prop~6.2]{Ko14}, which is much simpler for $ X = \emptyset $: for $ L \in I^n $, one can obtain $ B_L \in \sum_{J \in \mathcal{J}} \UU_{\Theta}^{0'} B_J $  by an induction on $ n = \wt(L) $ and using the $(q,\pi)$-Serre relations. We thus have that $\{ B_J \,|\, J \in \mathcal{J} \}$ spans $ \UUi $. The fact that  $ \{ B_J \,|\, J \in \mathcal{J} \} $ is linearly independent follows from the specific form of the generators $ B_i $ having `leading term' $ F_i $ and the triangular decomposition.
	\end{proof}

	\section{$\imath^\pi$-divided powers and expansion formulas in rank one}
	\label{section:iDP}
	
	In this section we will describe the $ \imath^\pi$-divided powers, which are generalizations of the formulas for $ \imath $-divided powers developed in \cite{BeW18} to the quantum covering group setting.
	
	Recall from \cite[2.1]{CHW13} that the rank one quantum covering group $ \UU $ with a single odd root (i.e.  $ \UU $ is of type $ I = I_{\odd} = \{1\}$) is the $\K(q)^\pi$-algebra generated by $E, F, K^{\pm 1}, J $, subject to the relations:  $KK^{-1}=K^{-1}K=1$, and 
	\begin{align}
	\label{eq:osp12}
	\begin{split}
	JK = KJ, \quad JE = EJ, &\quad JF = FJ, \quad J^2 = 1, \\
	K E K^{-1} = q^2 E K, &\quad K F K^{-1} = q^{-2} FK, \\
	EF - \pi FE &=\frac{JK-K^{-1}}{\pi q-q^{-1}}.
	\end{split}
	\end{align}
	The rank one $\imath${}quantum covering group $ \UUi $ is generated as a $ \K(q)^\pi $-algebra by a single generator $$ \tfk = F + q^{-1} E K^{-1}. $$ 
	
	\begin{lem}
		\label{lem:vs}
		There is an anti-involution $\vs$ of the $\K$-algebra $\UU$ fixing the generators $ E, F, K^{\pm 1}, J $ and sending $ q \mapsto q^{-1} $.
	\end{lem}
	
	\begin{proof}
		We have
		\begin{equation*}
		\vs(K E K^{-1}) = K^{-1} E K = q^{-2} E = \vs(q^2 E), \quad \vs(K F K^{-1}) = K^{-1} F K = q^2 F = \vs(q^{-2} F).
		\end{equation*}
		
		We also have
		\begin{equation*}
		\vs(EF - \pi FE) = FE - \pi EF =  {JK - K^{-1} \over \pi q^{-1} - q} = \vs \bigg( {JK - K^{-1} \over \pi q - q^{-1}}, \bigg)
		\end{equation*}
		and so $ \vs $ preserves all the relations in \eqref{eq:osp12} (since $ J $ is central).
	\end{proof}
	
	Note that $ \vs([n]) = \pi^{n-1} [n] $, and so $ \vs{[n]^!} = \pi^{{n \choose 2}} [n] $.
	\subsection{The algebra $\UUdot$ in rank one}
	
	Denote by $\UUdot$ the modified quantum group of $\osp(1|2)$, as the odd rank one case of \S\ref{subsec:UUdot}. 
	
	Let $\UAdot$ be the $\A$-subalgebra of $\UUdot$  generated by $E^{(n)} \one_\la, F^{(n)} \one_\la, \one_\la$, for all $n\ge 0$ and $\la\in \Z$.
	There is a natural left action of $\UU$ on $\UUdot$ such that $K \one_\la =q^\la \one_\la$ and $ J \one_\la = \pi^\la \one_\la $. 
	Denote by 
	\[
	\UAdot_\ev =\bigoplus_{\la\in \Z} \, \UAdot \one_{2\la},
	\qquad
	\UAdot_\od =\bigoplus_{\la\in \Z}\, \UAdot \one_{2\la -1}. 
	\]
	We have $\UAdot =\UAdot_\ev  \oplus \UAdot_\od$. By a base change we define  $\UUdot_\ev$ and $\UUdot_\od$ accordingly so that $\UUdot =\UUdot_\ev  \oplus \UUdot_\od$.

	\subsection{Recursive definition and closed form formulas}
	We have the following generalizations of the formulas for $ \imath{} $divided powers developed in \cite{BeW18}: the even $\imath^\pi$-divided powers $\dvev{n}$ satisfy and are in turn determined by the  following recursive relations: 
	\begin{align}   
	\label{eq:dvev-rec}
	\begin{split}
	\tfk \cdot \dvev{2a-1} &=[2a] \dvev{2a},
	\\
	\tfk \cdot \dvev{2a} &=  [2a+1] \dvev{2a+1} + [2a] \blue{J} \dvev{2a-1}, \quad \text{ for } a\ge 1.
	\end{split}
	\end{align}
	where $ [n]: = [n]_{q,\pi} $ here denotes the $ (q,\pi) $-integer; for the remainder of this section these subscripts will be suppressed. 
	
	Analogously, the odd $\imath{}$divided powers $\dvo{n}$ satisfy (and are determined by) the following recursive relations: 
	\begin{align}  
	\label{eq:dvo-rec}
	\begin{split}
	\tfk \cdot \dvo{2a} &=  [2a+1] \dvo{2a+1} ,
	\\
	\tfk \cdot \dvo{2a+1} &=[2a+2] \dvo{2a+2} + [2a+1] \blue{\pi J} \dvo{2a}, \quad \text{ for } a \ge 0.
	\end{split}
	\end{align}
	
	Solving these recursive formulas, we arrive at the following closed form formulas:
	\begin{align}
	\label{def:dvev}
	\begin{split}
	\dvev{2a} & = {\tfk^2 (\tfk^2 - [2]^2 \blue{J}) \cdots (\tfk^2-[2a-4]^2 \blue{J}) (\tfk - [2a-2]^2 \blue{J}) \over [2a]^!},\\
	\dvev{2a+1} & = {  \tfk^2 (\tfk^2 - [2]^2 \blue{J}) \cdots (\tfk^2-[2a-2]^2 \blue{J}) (\tfk - [2a]^2 \blue{J})  \over [2a+1]!},
	\quad \text{ for } a\ge 0,
	\end{split}
	\end{align}
	and
	\begin{align}
	\label{eq:dvo}
	\begin{split}
	\dvo{2a} &= { (\tfk^2 -  \blue{\pi J} ) (\tfk^2 - \blue{\pi} [3]^2 \blue{J}) \cdots (\tfk - \blue{\pi} [2a-1]^2 \blue{J}) \over [2a]!},\\
	\dvo{2a+1} &= { \tfk (\tfk^2 - \blue{\pi J} ) (\tfk^2 - \blue{\pi} [3]^2 \blue{J}) \cdots (\tfk - \blue{\pi} [2a-1]^2 \blue{J}) \over [2a+1]!},
	\quad \text{ for } a\ge 0.
	\end{split}
	\end{align}
	
	For example, $\dvev{0}=1$, $\dvev{1} =\tfk$, $\dvev{2} =\tfk^2/[2]$, and $\dvev{3} = \tfk(\tfk^2-J [2]^2)/[3]!$, and $\dvo{0}=1, \dvo{1}=\tfk, \dvo{2} =(\tfk^2-\pi J)/[2]$ and $\dvo{3} =\tfk(\tfk^2-\pi J)/[3]!$. 
	
	\subsection{Expansion formulas}
	In this subsection we will formulate a number of useful expansion formulas for $ \dvev{n} $ and $ \dvo{n} $, cf. \cite{BeW18}. We set 
	\begin{equation}  \label{eq:t}
	\Y :=q^{-1}EK^{-1}, 
	\qquad
	\kk:=\frac{K^{-2}-\blue{J}}{q^2-\blue{\pi}}, \qquad \tfk:=\Y +F.
	\end{equation}
	
	Define, for $a\in \Z, n\ge 0$, 
	\begin{equation}  \label{kbinom}
	\qbinom{\kk;a}{n} =\prod_{i=1}^n \frac{q^{4a+4i-4} K^{-2} - \blue{J}}{q^{4i} -1},
	\qquad 
	[\kk;a]= \qbinom{\kk;a}{1}.
	\end{equation}
	Note that 
	$\kk=q[2]\, [\kk;0].$ 

	It follows from \eqref{eq:osp12} that, for $a\in \Z$ and $n\ge 0$,  
	\begin{align}
	\label{FEk}
	F \Y  = \kk + \pi q^{-2} \Y  F,
	\qquad
	\qbinom{\kk;a}{n} F = F \qbinom{\kk;a+1}{n},
	\qquad
	\qbinom{\kk;a}{n} \Y  =\Y  \qbinom{\kk;a-1}{n}.
	\end{align}
	
	Also define for $a\in \Z, n\ge 1$, 
	\begin{equation}   \label{brace}
	\LR{\kk;a}{0}=1, 
	\qquad
	\LR{\kk;a}{n}= \prod_{i=1}^n \frac{q^{4a+4i-4} K^{-2}- \blue{\pi} q^2 \blue{J}}{q^{4i}-1}, 
	\qquad \llbracket \kk;a\rrbracket = \LR{\kk;a}{1}.
	\end{equation}
	Note $\kk=q[2] \llbracket \kk;0 \rrbracket+1$. 
	It follows from \eqref{eq:osp12} and \eqref{brace} that, for $n\ge 0$ and $a\in \Z$, 
	\begin{align}
	\LR{\kk;a}{n} F = F  \LR{\kk;a+1}{n},
	\qquad
	\LR{\kk;a}{n} \Y =\Y \LR{\kk;a-1}{n}.
	\end{align}
	
	Just as in the even case, we also have
	\begin{equation}
	\label{kqbinom2}
	\LR{\kk;a}{n} \one_{2\la-1} = q^{2n(a-\la)} \qbinom{a-\la-1+n}{n}_{q^2} \one_{2\la-1} \in \UAdot_\od. 
	\end{equation}
	
	\begin{lem}
		\label{dpY}
		For $n\in \N$, we have 
		\[
		\Y^{(n)}= q^{-n^2} E^{(n)} K^{-n}.
		\]
	\end{lem}
	
	\begin{proof}
		Follows by induction on $n$, using \eqref{eq:osp12} and \eqref{eq:t}.
	\end{proof}
	
	\begin{lem}
		The following formula holds for $n\ge 0$:  
		\begin{align}   \label{FYn}
		F \Y^{(n)} &=(\blue{\pi} q^{-2})^n \Y^{(n)} F +  \Y^{(n-1)} \frac{q^{3-3n} K^{-2} - (\blue{\pi} q)^{1-n} J }{q^2-\blue{\pi}}.
		\end{align} 
	\end{lem}
	
	\begin{proof}
		We shall prove the following equivalent formula  by induction on $n$:
		\[
		F\Y^n = (\pi q^{-2})^n \Y^n F + (q^2-\pi)^{-1} [n] \Y^{n-1} \big( q^{3-3n} K^{-2} - (\pi q)^{1-n} J \big).
		\]
		The base case when $n=1$ is covered by \eqref{FEk}. 
		Assume the formula is proved for $F\Y^n$.
		Then by inductive assumption we have
		\begin{align*}
		F\Y^{n+1} &= (\pi q^{-2})^n \Y^n F\Y + (q^2-\pi)^{-1} [n] \Y^{n-1} \big( q^{3-3n} K^{-2} - (\pi q)^{1-n} J \big) \Y \\
		&=(\pi q^{-2})^n \Y^n (\pi q^{-2} \Y F + (q^2-\pi)^{-1} \big( K^{-2} - J\big)  ) + (q^2-\pi)^{-1} [n] \Y^{n} \big( q^{-1-3n} K^{-2} - (\pi q)^{1-n} J \big) \\
		&=(\pi q^{-2})^{n+1} \Y^{n+1} F + (q^2-\pi)^{-1} [n+1 ] \Y^n \big( q^{-3n} K^{-2} - (\pi q)^{-n} J \big),
		\end{align*}
		since $ [n+1] = (\pi q)^n + q^{-1} [n] = \pi q [n] + q^{-n} $. The lemma is proved. 
	\end{proof}
	
	For $n\in \N$, we denote
	\begin{align}  \label{eq:divB}
	b_\pi^{(n)}  = \sum_{a=0}^n  (\blue{\pi} q)^{-a(n-a)}  \Y^{(a)}F^{(n-a)}.
	\end{align}
	
	\subsection{The $\Y \kk F$-formula for $\dvev{n}$} 
	
	Recall $\qbinom{\kk;a}{n}$ from \eqref{kbinom}.
	
	\begin{example}
		We computed the following examples of $\dvev{n}$, for $2\le n\le 4$:
		\begin{align*}
		\dvev{2} &= {\tfk^2 \over [2]} = b_\pi^{(2)} + \blue{\pi} q [\kk;0],
		\\
		\dvev{3} &= {\tfk^3 - J [2]^2 \tfk \over [3]^!} = b_\pi^{(3)} + \blue{\pi} q^3[\kk;-1]F + \blue{\pi} q^3 \Y [\kk;-1],
		\\
		\dvev{4} &= {\tfk^4 - J [2]^2 \tfk^2 \over [4]^!} = b_\pi^{(4)}  + \blue{\pi} q\Y^{(2)}  [\kk;-1] + \blue{\pi} q  [\kk;-1] F^{(2)} + \Y [\kk; -1] F + q^{6}  \qbinom{\kk;-1}{2}.
		\end{align*}
	\end{example}

	\begin{thm}  \label{thm:iDP:ev}
		For $m\ge 1$, we have
		\begin{align}
		\dvev{2m} &= \sum_{c=0}^m \sum_{a=0}^{2m-2c} (\blue{\pi} q)^{\binom{2c}{2} -a(2m-2c-a)} \Y^{(a)}  \qbinom{\kk;1-m}{c}  F^{(2m-2c-a)}, 
		\label{t2m}
		\\
		\dvev{2m-1} &=  \sum_{c=0}^{m-1} \sum_{a=0}^{2m-1-2c} (\blue{\pi} q)^{\binom{2c+1}{2} -a(2m-1-2c-a)} \Y^{(a)}  \qbinom{\kk; 1-m}{c}  F^{(2m-1-2c-a)}. 
		\label{t2m-1}
		\end{align}
	\end{thm}
	
	
	\begin{proof}
		We prove the formulae for $\dvev{n}$ by using the recursive relations \eqref{eq:dvev-rec} and induction on $n$. The base cases for $n=1,2$ are clear.
		The induction is carried out in 2 steps. 
		
		\vspace{.3cm}
		(1) First by assuming the formula for $\dvev{2m-1}$ in \eqref{t2m-1}, we shall establish the formula \eqref{t2m} 
		for $\dvev{2m}$, via the identity $[2m] \dvev{2m} =\tfk \cdot \dvev{2m-1}$ in \eqref{eq:dvev-rec}. 
		
		Recall the formula \eqref{t2m-1} for $\dvev{2m-1}$.
		Using $\tfk=\Y+F$ and applying \eqref{FYn} to $F\Y^{(a)}$ we have
		\begin{align}  \label{tt2m-1}
		\tfk\cdot \dvev{2m-1} 
		&= \sum_{c=0}^{m-1} \sum_{a=0}^{2m-1-2c} (\pi q)^{\binom{2c+1}{2} -a(2m-1-2c-a)}  \tfk \Y^{(a)}  \qbinom{\kk; 1-m}{c}  F^{(2m-1-2c-a)}
		\\
		&= \sum_{c=0}^{m-1} \sum_{a=0}^{2m-1-2c} (\pi q)^{\binom{2c+1}{2} -a(2m-1-2c-a)} \cdot 
		\notag \\
		&\qquad \quad \left( \Y \Y^{(a)}  +(\pi q^{-2})^a \Y^{(a)} F  +\Y^{(a-1)} \frac{q^{3-3a} K^{-2} - (\pi q)^{1-a} J }{q^2-\pi} \right) \qbinom{\kk; 1-m}{c}  F^{(2m-1-2c-a)}  
		\notag \\
		&= \sum_{c=0}^{m-1} \sum_{a=0}^{2m-1-2c} (\pi q)^{\binom{2c+1}{2} -a(2m-1-2c-a)} \cdot 
		\notag \\
		& \left( [a+1] \Y^{(a+1)}  \qbinom{\kk; 1-m}{c}  F^{(2m-1-2c-a)}
		+(\pi q^{-2})^a [2m-2c-a] \Y^{(a)}  \qbinom{\kk; -m}{c}  F^{(2m-2c-a)} \right.
		\notag \\
		& \qquad +
		\left. \Y^{(a-1)} \frac{q^{3-3a} K^{-2} - (\pi q)^{1-a} J }{q^2-\pi} \qbinom{\kk; 1-m}{c}  F^{(2m-1-2c-a)} \right). 
		\notag
		\end{align}
		
		We reorganize the formula \eqref{tt2m-1} in the following form
		\[
		[2m] \cdot \dvev{2m}=
		\tfk \cdot \dvev{2m-1} =  \sum_{c=0}^m \sum_{a=0}^{2m-2c} \Y^{(a)} f_{a,c}(\kk) F^{(2m-2c-a)}, 
		\]
		where
		\begin{align*}
		f^\pi_{a,c}(\kk) &= (\pi q)^{\binom{2c+1}{2} -(a-1)(2m-2c-a)}   [a]\,  \qbinom{\kk; 1-m}{c} \\
		&\quad + \left( \pi^a (\pi q)^{\binom{2c+1}{2} -a(2m-1-2c-a) -2a} [2m-2c-a]\, \qbinom{\kk; -m}{c} \right. \\
		&\qquad\quad \left. +q^{\binom{2c-1}{2} -(a+1)(2m-2c-a)}  \frac{q^{-3a} K^{-2} -(\pi q)^{-a} J }{q^2-\pi} \qbinom{\kk; 1-m}{c-1} \right).
		\end{align*}
		A direct computation gives us
		\begin{align*}
		f^\pi_{a,c}(\kk) &=  (\pi q)^{\binom{2c}{2} - a(2m-2c-a)} (\pi q)^{2m-a}  [a] \,  \qbinom{\kk; 1-m}{c} +(\pi q)^{\binom{2c}{2}- a(2m-2c-a)} \cdot  \\
		& \qquad \cdot \left( \pi^a (\pi q)^{2c-a}[2m-2c-a] \frac{q^{-4m} K^{-2}-J}{q^{4c}-1}  + (\pi q)^{1+a-2m} \frac{q^{-3a} K^{-2}-(\pi q)^{-a}}{q^{2}- \pi} \right)  \qbinom{\kk; 1-m}{c-1}
		\\
		&=  (\pi q)^{\binom{2c}{2} - a(2m-2c-a)} (\pi q)^{2m-a}  [a] \,  \qbinom{\kk; 1-m}{c} +(\pi q)^{\binom{2c}{2}- a(2m-2c-a)} \cdot  \\
		& \qquad \cdot \left( \pi^a (\pi q)^{2c-a}[2m-2c-a] \frac{q^{-4m} K^{-2}-J}{q^{4c}-1}  + (\pi q)^{2c+a-2m} [2c] \frac{q^{-3a} K^{-2}-(\pi q)^{-a}}{q^{4c}- 1} \right)  \qbinom{\kk; 1-m}{c-1}
		\\
		&= (\pi q)^{\binom{2c}{2} - a(2m-2c-a)} (\pi q)^{2m-a}  [a] \,  \qbinom{\kk; 1-m}{c}
		+(\pi q)^{\binom{2c}{2} - a(2m-2c-a)} q^{-a} [2m-a] \,  \qbinom{\kk; 1-m}{c}
		\\
		&=  (\pi q)^{\binom{2c}{2} - a(2m-2c-a)} ((\pi q)^{2m - a} [a] + q^{-a} [2m - a]) \,  \qbinom{\kk; 1-m}{c}
		\\
		&=  (\pi q)^{\binom{2c}{2} - a(2m-2c-a)}  [2m]\,  \qbinom{\kk; 1-m}{c}.
		\end{align*}
		Hence we have obtained the formula \eqref{t2m}  for $\dvev{2m}$.
		
		\vspace{.3cm}
		(2) 
		Now by assuming the formula for $\dvev{2m}$ in \eqref{t2m}, we shall establish the following formula (with $m$ in \eqref{t2m-1} replaced by $m+1$)
		\begin{align}
		\dvev{2m+1} &=  \sum_{c=0}^m \sum_{a=0}^{2m+1-2c} (\pi q)^{\binom{2c+1}{2} -a(2m+1-2c-a)} \Y^{(a)}  \qbinom{\kk;-m}{c}  F^{(2m+1-2c-a)}. 
		\label{t2m+1}
		\end{align}
		
		Recall the formula for $\dvev{2m}$ in \eqref{t2m}.
		Using  $\tfk=\Y+F$ and applying \eqref{FYn} to $F\Y^{(a)}$ we have
		\begin{align*}
		\tfk\cdot \dvev{2m} &=   \sum_{c=0}^m \sum_{a=0}^{2m-2c} (\pi q)^{\binom{2c}{2} -a(2m-2c-a)} \tfk \Y^{(a)}  \qbinom{\kk; 1-m}{c}  F^{(2m-2c-a)} \\
		&=  \sum_{c=0}^m \sum_{a=0}^{2m-2c} (\pi q)^{\binom{2c}{2} -a(2m-2c-a)}  \cdot \notag
		\\
		&\qquad \cdot \left( \Y \Y^{(a)}   +(\pi q^{-2})^a \Y^{(a)} F   +\Y^{(a-1)} \frac{q^{3-3a} K^{-2} - (\pi q)^{1-a} J }{q^2-\pi} \right) \qbinom{\kk; 1-m}{c}  F^{(2m-2c-a)}.
		\end{align*}
		We rewrite this as
		\begin{align}  \label{tt2m}
		\tfk\cdot \dvev{2m} 
		&=  \sum_{c=0}^m \sum_{a=0}^{2m-2c} (\pi q)^{\binom{2c}{2} -a(2m-2c-a)} \cdot 
		\left ( [a+1] \Y^{(a+1)}  \qbinom{\kk; 1-m}{c}  F^{(2m-2c-a)} \right.  \\
		&\qquad\qquad\qquad 
		+(\pi q^{-2})^a [2m+1-2c-a] \Y^{(a)}  \qbinom{\kk; -m}{c}  F^{(2m+1-2c-a)}  
		\notag \\
		& \qquad\qquad\qquad  \left. +\Y^{(a-1)} \frac{q^{3-3a} K^{-2} - (\pi q)^{1-a} J }{q^2-\pi} \qbinom{\kk; 1-m}{c}  F^{(2m-2c-a)} \right).
		\notag
		\end{align}
		We shall use \eqref{eq:dvev-rec}, \eqref{tt2m} and \eqref{t2m-1} to obtain a formula of the form
		\begin{equation} \label{eq:ttsum}
		[2m+1] \dvev{2m+1} =\tfk \cdot \dvev{2m} - [2m] J \dvev{2m-1}
		=   
		\sum_{c=0}^m \sum_{a=0}^{2m+1-2c} \Y^{(a)} g^\pi_{a,c}(\kk) F^{(2m+1-2c-a)}, 
		\end{equation}
		for some suitable $g^\pi_{a,c}(\kk)$. Then we have
		\begin{align*}
		g^\pi_{a,c}(\kk) &= (\pi q)^{\binom{2c}{2} -(a-1)(2m+1-2c-a)}   [a]\,  \qbinom{\kk; 1-m}{c} \\
		&\quad + \pi^a (\pi q)^{\binom{2c}{2} -a(2m-2c-a) -2a} [2m+1-2c-a]\, \qbinom{\kk; -m}{c}   \\
		&\quad + (\pi q)^{\binom{2c-2}{2} -(a+1)(2m+1-2c-a)}  \frac{q^{-3a} K^{-2} - (\pi q)^{-a} J }{q^2-\pi} \qbinom{\kk; 1-m}{c-1} 
		\\
		&\quad  -(\pi q)^{\binom{2c-1}{2} - a(2m+1-2c-a)} [2m] \qbinom{\kk; 1-m}{c-1} 
		\\
		&= \pi^a (\pi q)^{\binom{2c+1}{2} - a(2m+1-2c-a)} (\pi q)^{-2c-a} [2m+1-2c-a]\, \qbinom{\kk; -m}{c}   + (\pi q)^{\binom{2c+1}{2} - a(2m+1-2c-a)}  X, 
		\end{align*}
		where  
		\begin{align*}
		X&=(\pi q)^{2m+1-4c-a}   [a]\,  \qbinom{\kk; 1-m}{c}  
		\\
		&\qquad +(\pi q)^{-2m+a-4c+2} \frac{q^{-3a} K^{-2} - (\pi q)^{-a} J }{q^2-\pi} \qbinom{\kk; 1-m}{c-1} 
		-(\pi q)^{1-4c} [2m] J \qbinom{\kk; 1-m}{c-1}. 
		\end{align*}
		A direct computation allows us to simplify the expression for $X$ as follows:
		\begin{align*}
		X&= \big( (\pi q)^{2m+1-4c-a}  [a]  \frac{q^{4c-4m} K^{-2} - J}{q^{4c}-1} \\
		&\qquad +(\pi q)^{-2m+a-4c+2} \frac{q^{-3a} K^{-2} - (\pi q)^{-a} J }{q^2-\pi}  -(\pi q)^{1-4c} [2m]
		\big) \qbinom{\kk; 1-m}{c-1}
		\\
		&= (\pi q)^{2m-2c-a+1}  [2c+a]   \frac{q^{-4m} K^{-2} - J}{q^2-1}   \qbinom{\kk; 1-m}{c-1}
		\\
		&= (\pi q)^{2m-2c-a+1}  [2c+a]    \qbinom{\kk; -m}{c}. 
		\end{align*}
		
		Hence, we  obtain
		\begin{align*}
		g^\pi_{a,c}(\kk) &= \pi^a (\pi q)^{\binom{2c+1}{2} - a(2m+1-2c-a)}  (\pi q)^{-2c-a} [2m+1-2c-a]\, \qbinom{\kk; -m}{c} 
		\\
		&\qquad + (\pi q)^{\binom{2c+1}{2} - a(2m+1-2c-a)}  (\pi q)^{2m-2c-a+1}  [2c+a]    \qbinom{\kk; -m}{c}
		\\ &=  (\pi q)^{\binom{2c+1}{2} - a(2m+1-2c-a)}  [2m+1]\,  \qbinom{\kk; -m}{c}.
		\end{align*}
		Recalling the identity \eqref{eq:ttsum}, we have thus proved the formula \eqref{t2m+1}  for $\dvev{2m+1}$, and
		hence completed the proof of Theorem~\ref{thm:iDP:ev}. 
	\end{proof}
	
	\subsection{Reformulations of the expansion formulas for $ \dvev{n} $} 
	
	We can apply the anti-involution $ \vs $ in Lemma~\ref{lem:vs} to the formulas in Theorem~\ref{thm:iDP:ev} to obtain the following $F\kk\Y$-expansion formulas (cf. \cite[Prop~2.7]{BeW18}): 
	
	\begin{prop}  \label{iDP:FkY}
		For $m\ge 1$, we have
		\begin{align*}
		\dvev{2m} &= \sum_{c=0}^m \sum_{a=0}^{2m-2c} (-1)^c q^{3c+ a(2m-2c-a)} F^{(a)}  \qbinom{\kk;m-c}{c}  \Y^{(2m-2c-a)}, 
		\\
		\dvev{2m-1} &=  \sum_{c=0}^{m-1} \sum_{a=0}^{2m-1-2c} (-1)^c q^{c+ a(2m-1-2c-a)} F^{(a)}  \qbinom{\kk; m-c}{c}  \Y^{(2m-1-2c-a)}. 
		\end{align*}
	\end{prop}
	
	\begin{proof}
		The involution $ \vs $ in Lemma~\ref{lem:vs} fixes $ F, \Y, J, K^{-1} $ and sends
		\begin{equation*}
		\dvev{n} \mapsto \pi^{{n \choose 2}} \dvev{n}, \qquad \qbinom{\kk;a}{n} \mapsto (-1)^n q^{2n(n+1)} \qbinom{\kk;1-a-n}{n}, \quad \forall a \in \Z, n \in \N.
		\end{equation*}	
		
		Applying $ \vs $ to \eqref{t2m}, we end up with $ \pi^{{2m \choose 2}} $ on the LHS and $ \pi^{{a \choose 2} + {2m - 2c -a \choose 2}} $ on the RHS. Dividing through by $ \pi^{{2m \choose 2}} $, we see that the powers of $ \pi $ inside the double sum work out to 
		$$ 
		\pi^{{2m - 2c -a \choose 2} + {a \choose 2} - {2m \choose 2}} \pi^{ {2c \choose 2} + a} = \pi^{a+c} \pi^{c+a} = 1 .
		$$ Similarly for the odd power case \eqref{t2m-1}, the powers of $ \pi $ in the double sum work out to $ \pi^{c + a - a} \pi^c = 1 $. Thus, both formulas are identical to the non-super case in \cite[Prop~2.7]{BeW18}.
	\end{proof}

	For $\la \in \Z$, 
	\begin{equation}
	\label{k=qbinom}
	\qbinom{\kk;a}{n} \one_{2\la} = q^{2n(a-1-\la)} \qbinom{a-1-\la+n}{n}_{q^2} \one_{2\la} \in \UAdot_\ev,
	\end{equation}
	even though $\qbinom{\kk;a}{n}$ does not lie in $\UA$ in general (cf. \cite{BeW18}).
	
	Thus, by the same argument as \cite[Prop~2.8]{BeW18}, we have the following reformulation of Theorem~\ref{thm:iDP:ev}; the only difference here is the factor of $\blue{\pi^a}$, which comes from Lemma~\ref{dpY}):
	
	\begin{prop}
		\label{prop:iDPevdot}
		For $m\ge 1$ and $\la \in \Z$, we have
		\begin{align}
		\dvev{2m} \one_{2\la}
		&= \sum_{c=0}^m \sum_{a=0}^{2m-2c} \blue{\pi^a} (\blue{\pi} q)^{2(a+c)(m-a-\la)-2ac-\binom{2c+1}{2}} \qbinom{m-c-a-\la}{c}_{q^2}
		E^{(a)}  F^{(2m-2c-a)}\one_{2\la}, 
		\label{t2mdot}
		\\
		\dvev{2m-1} \one_{2\la}
		&=  \sum_{c=0}^{m-1} \sum_{a=0}^{2m-1-2c}  
		\label{t2m-1dot}\\
		&  
		\quad \blue{\pi^a} (\blue{\pi} q)^{2(a+c)(m-a-\la)-2ac-a-\binom{2c+1}{2}}  \qbinom{m-c-a-\la-1}{c}_{q^2}  E^{(a)}  F^{(2m-1-2c-a)}\one_{2\la}. 
		\notag
		\end{align}
		In particular, we have $\dvev{n}\one_{2\la} \in \UAdot_\ev$, for all $n\in \N$. 
	\end{prop}

	\subsection{The $\Y \kk F$-formula for $\dvo{n}$}
	
	Recall that $ \llbracket \kk;0 \rrbracket = \LR{\kk;0}{1} $.
	
	\begin{example}
		We have the following examples of $\dvo{n}$, for $2\le n\le 4$:
		\begin{align*}
		\dvo{2} &= {\tfk^2 - \pi J \over [2]!}
		= b_\pi^{(2)} + \pi q\llbracket \kk;0 \rrbracket,
		\\
		\dvo{3} &= {\tfk^3 - \pi J \tfk \over [3]!}
		= b_\pi^{(3)} + \pi q^{-1} \llbracket \kk;0 \rrbracket F + \pi q^{-1} \Y \llbracket \kk;0 \rrbracket,
		\\
		\dvo{4} &= {(\tfk^2 - \pi J [3]^2)(\tfk^2 - \pi J) \over [4]!}
		= b_\pi^{(4)} +  \pi q \Y^{(2)} \llbracket \kk;-1 \rrbracket + \pi q \llbracket \kk;-1 \rrbracket F^{(2)} + \Y \llbracket \kk;-1 \rrbracket F + q^6 \LR{\kk;-1}{2}.
		\end{align*}
	\end{example}
	
	\begin{thm}  \label{thm:iDP:odd}
		For $m\ge 0$, we have
		\begin{align}
		\dvo{2m} &= \sum_{c=0}^m \sum_{a=0}^{2m-2c} (\pi q)^{\binom{2c}{2} -a(2m-2c-a)}
		\Y^{(a)}  \LR{\kk;1-m}{c}  F^{(2m-2c-a)}, 
		\label{todd2m}  
		\\
		\dvo{2m+1} &=  \sum_{c=0}^{m} \sum_{a=0}^{2m+1-2c} (\pi q)^{\binom{2c-1}{2}-1 -a(2m+1-2c-a)}
		\Y^{(a)}  \LR{\kk;1-m}{c}  F^{(2m+1-2c-a)}. 
		\label{todd2m+1}   
		\end{align}
	\end{thm}
	

	\begin{proof}
		\label{sec:proof:iDP:odd}
		
		As in \cite{BeW18}, we prove the formulae for $\dvo{n}$ by induction on $n$. The base case for $n=1$ is clear.
		The induction is carried out in 2 steps. 
		
		\vspace{.3cm}
		(1) First by assuming the formula for $\dvo{2m}$ in \eqref{todd2m}, we shall establish the formula \eqref{todd2m+1} 
		for $\dvo{2m+1}$, via the identity $[2m+1] \dvo{2m+1} = \tfk \cdot \dvo{2m}$ in \eqref{eq:dvo-rec}. 
		
		Recall the formula \eqref{todd2m} for $\dvo{2m}$.
		Using $\tfk=\Y+F$ and applying \eqref{FYn} to $F\Y^{(a)}$ we have
		\begin{align}  \label{ttodd2m}
		\tfk\cdot \dvo{2m} 
		&= \sum_{c=0}^{m} \sum_{a=0}^{2m-2c} (\pi q)^{\binom{2c}{2} -a(2m-2c-a)}  \tfk \Y^{(a)}  \LR{\kk;1-m}{c} F^{(2m-2c-a)}
		\\
		&= \sum_{c=0}^{m} \sum_{a=0}^{2m-2c} (\pi q)^{\binom{2c}{2} -a(2m-2c-a)} \cdot 
		\notag \\
		&\qquad \quad \left( \Y \Y^{(a)}  + (\pi q^{-2})^a \Y^{(a)} F  +\Y^{(a-1)} \frac{q^{3-3a} K^{-2} - (\pi q)^{1-a} J }{q^2-\pi} \right) \LR{\kk;1-m}{c} F^{(2m-2c-a)}  
		\notag \\
		&= \sum_{c=0}^{m} \sum_{a=0}^{2m-2c} (\pi q)^{\binom{2c}{2} -a(2m-2c-a)} \cdot 
		\notag \\
		& \left( [a+1] \Y^{(a+1)}  \LR{\kk;1-m}{c} F^{(2m-2c-a)}
		+(\pi q^{-2})^a [2m+1-2c-a] \Y^{(a)}  \LR{\kk;-m}{c} F^{(2m+1-2c-a)} \right.
		\notag \\
		& \qquad +
		\left. \Y^{(a-1)} \frac{q^{3-3a} K^{-2} - (\pi q)^{1-a} J }{q^2-\pi} \LR{\kk;1-m}{c} F^{(2m-2c-a)} \right). 
		\notag
		\end{align}

		We  reorganize the formula \eqref{ttodd2m} in the following form
		\[
		[2m+1] \dvo{2m+1} = \tfk \cdot \dvo{2m} =  \sum_{c=0}^m \sum_{a=0}^{2m+1-2c} \Y^{(a)} \texttt{f}^\pi_{a,c}(\kk) F^{(2m+1-2c-a)}, 
		\]
		where
		\begin{align*}
		\texttt{f}^\pi_{a,c}(\kk) &= (\pi q)^{\binom{2c}{2} -(a-1)(2m+1-2c-a)}   [a]\,  \LR{\kk;1-m}{c}\\
		&\quad + \left( \pi^a (\pi q)^{\binom{2c}{2} -a(2m-2c-a) -2a} [2m+1-2c-a]\, \LR{\kk;-m}{c}\right. \\
		&\qquad\quad \left. + (\pi q)^{\binom{2c-2}{2} -(a+1)(2m+1-2c-a)}  \frac{q^{-3a} K^{-2} - (\pi q)^{-a} J }{q^2-\pi} \LR{\kk;1-m}{c-1}\right).
		\end{align*}
		A direct computation gives us
		\begin{align*}
		\texttt{f}^\pi_{a,c}(\kk) &=  (\pi q)^{\binom{2c-1}{2} -1 - a(2m+1-2c-a)} (\pi q)^{2m+1-a}  [a] \,  \LR{\kk;1-m}{c}+(\pi q)^{\binom{2c-1}{2} -1 - a(2m+1-2c-a)} \cdot  \\
		&\quad\qquad \cdot \bigg( \pi^a (\pi q)^{2c-a}[2m+1-2c-a] \frac{q^{-4m} K^{-2}-\pi q^2 J}{q^{4c}-1} \\
		&\quad\qquad\qquad + (\pi q)^{2+a-2m} \frac{q^{-3a} K^{-2}-(\pi q)^{-a}}{q^{2}-\pi}  \bigg)  \LR{\kk;1-m}{c-1}
		\\
		&=  (\pi q)^{\binom{2c-1}{2} -1 - a(2m+1-2c-a)} (\pi q)^{2m+1-a}  [a] \,  \LR{\kk;1-m}{c}
		\\
		&\qquad + (\pi q)^{\binom{2c-1}{2} -1 - a(2m+1-2c-a)} q^{-a} [2m+1-a] \,  \LR{\kk;1-m}{c}
		\\
		&= (\pi q)^{\binom{2c-1}{2} -1 - a(2m+1-2c-a)}  [2m+1]\,  \LR{\kk;1-m}{c}.
		\end{align*}
		Hence we have obtained the formula \eqref{todd2m+1}  for $\dvo{2m+1}$.
		
		\vspace{.3cm}
		(2) 
		Now by assuming the formula for $\dvo{2m+1}$ in \eqref{todd2m+1}, we shall establish the following formula (with $m$ in \eqref{todd2m} replaced by $m+1$)
		\begin{align}
		\dvo{2m+2} &=  \sum_{c=0}^{m+1} \sum_{a=0}^{2m+2-2c} (\pi q)^{\binom{2c}{2} -a(2m+2-2c-a)} \Y^{(a)}  \LR{\kk;-m}{c} F^{(2m+2-2c-a)}. 
		\label{todd2m+2}
		\end{align}
		
		Recall the formula \eqref{todd2m+1} for $\dvo{2m+1}$.
		Using  $\tfk=\Y+F$ and applying \eqref{FYn} to $F\Y^{(a)}$ we have
		\begin{align*}
		\tfk \cdot \dvo{2m+1} &=   \sum_{c=0}^m \sum_{a=0}^{2m+1-2c} (\pi q)^{\binom{2c-1}{2}-1 -a(2m+1-2c-a)}   \tfk \Y^{(a)}  \LR{\kk;1-m}{c} F^{(2m+1-2c-a)} \\
		&=  \sum_{c=0}^m \sum_{a=0}^{2m+1-2c} (\pi q)^{\binom{2c-1}{2}-1 -a(2m+1-2c-a)}  \cdot \notag
		\\
		&\qquad \cdot \left( \Y \Y^{(a)}   +(\pi q^{-2})^a \Y^{(a)} F   +\Y^{(a-1)} \frac{q^{3-3a} K^{-2} - (\pi q)^{1-a} J }{q^2-\pi} \right) \LR{\kk;1-m}{c} F^{(2m+1-2c-a)}.
		\end{align*}
		We rewrite this as
		\begin{align}  \label{ttodd2m+1}
		\tfk \cdot \dvo{2m+1}
		&=  \sum_{c=0}^m \sum_{a=0}^{2m+1-2c} (\pi q)^{\binom{2c-1}{2}-1 -a(2m+1-2c-a)} \cdot 
		\left ( [a+1] \Y^{(a+1)} \LR{\kk;1-m}{c} F^{(2m+1-2c-a)} \right.  \\
		&\qquad\qquad\qquad 
		+(\pi q^{-2})^a [2m+2-2c-a] \Y^{(a)}  \LR{\kk;-m}{c} F^{(2m+2-2c-a)}  
		\notag \\
		& \qquad\qquad\qquad  \left. +\Y^{(a-1)} \frac{q^{3-3a} K^{-2} - (\pi q)^{1-a} J }{q^2-\pi} \LR{\kk;1-m}{c} F^{(2m+1-2c-a)} \right).
		\notag
		\end{align}
		We shall use \eqref{eq:dvo-rec}, \eqref{ttodd2m+1} and \eqref{todd2m} to obtain a formula of the form
		\begin{equation} \label{eq:ttoddsum}
		[2m+2] \dvo{2m+1} = \tfk\cdot \dvo{2m+1} - \pi [2m+1] J \dvo{2m}
		=   
		\sum_{c=0}^{m+1} \sum_{a=0}^{2m+2-2c} \Y^{(a)} \texttt{g}^\pi_{a,c}(\kk) F^{(2m+2-2c-a)}, 
		\end{equation}
		for some suitable $\texttt{g}^\pi_{a,c}(\kk)$. Then we have
		\begin{align*}
		\texttt{g}^\pi_{a,c}(\kk) &= (\pi q)^{\binom{2c-1}{2}-1 -(a-1)(2m+2-2c-a)}   [a]\,  \LR{\kk;1-m}{c}\\
		&\quad + \pi^a (\pi q)^{\binom{2c-1}{2}-1 -a(2m+1-2c-a) -2a} [2m+2-2c-a]\, \LR{\kk;-m}{c} \\
		&\quad + (\pi q)^{\binom{2c-3}{2}-1 -(a+1)(2m+2-2c-a)}  \frac{q^{-3a} K^{-2} - (\pi q)^{-a} J }{q^2-\pi} \LR{\kk;1-m}{c-1}
		\\
		&\quad - (\pi q)^{\binom{2c-2}{2} - a(2m+2-2c-a)} [2m+1] \LR{\kk;1-m}{c-1}
		\\
		&= \pi^a (\pi q)^{\binom{2c}{2} - a(2m+2-2c-a)} (\pi q)^{-2c-a} [2m+2-2c-a]\, \LR{\kk;-m}{c}  + (\pi q)^{\binom{2c}{2} - a(2m+2-2c-a)}  \texttt{X}^\pi, 
		\end{align*}
		where  
		\begin{align*}
		\texttt{X}^\pi &= (\pi q)^{2m+2-4c-a}   [a]\,  \LR{\kk;1-m}{c}
		\\
		&\qquad + (\pi q)^{-2m+3-4c+a} \frac{q^{-3a} K^{-2} - (\pi q)^{-a} J }{q^2-\pi} \LR{\kk;1-m}{c-1}
		-(\pi q)^{3-4c} [2m+1] \LR{\kk;1-m}{c-1}.
		\end{align*}
		A direct computation allows us to simplify the expression for $\texttt{X}^\pi$ as follows:
		\begin{align*}
		\texttt{X}^\pi&= \bigg( (\pi q)^{2m+2-4c-a}   [a]  \frac{q^{4c-4m} K^{-2} - \pi q^2 J}{q^{4c}-1}
		\\&\quad\qquad + (\pi q)^{-2m+3-4c+a} \frac{q^{-3a} K^{-2} - (\pi q)^{-a} J}{q^2-\pi}  - (\pi q)^{3-4c} [2m+1]
		\bigg) \LR{\kk;1-m}{c-1}
		\\
		&= (\pi q)^{2m+2-2c-a}  [2c+a]   \frac{q^{-4m} K^{-2} - \pi q^2 J}{q^{4c}-1}   \LR{\kk;1-m}{c-1}
		\\
		&= (\pi q)^{2m+2-2c-a}  [2c+a]    \LR{\kk;-m}{c}. 
		\end{align*}
		
		Hence, we  obtain
		\begin{align*}
		\texttt{g}^\pi_{a,c}(\kk) &= (\pi q)^{\binom{2c}{2} - a(2m+2-2c-a)}  q^{-2c-a} [2m+2-2c-a]\, \LR{\kk;-m}{c}
		\\
		&\qquad + (\pi q)^{\binom{2c}{2} - a(2m+2-2c-a)}  (\pi q)^{2m+2-2c-a}  [2c+a]    \LR{\kk;-m}{c}
		\\ &=  (\pi q)^{\binom{2c}{2} - a(2m+2-2c-a)}  [2m+2]\,  \LR{\kk;-m}{c},
		\end{align*}
		where the last equality uses the general identity $ q^{-l} [k - 1] + (\pi q)^{k - 1} [l] = [k] $. Recalling the identity \eqref{eq:ttoddsum}, we have proved the formula \eqref{todd2m+2}  for $\dvo{2m+2}$, and
		hence completed the proof of Theorem~\ref{thm:iDP:odd}. 
	\end{proof}

	\subsection{Reformulation of the expansion formulas for $ \dvo{n} $}
	
	Just as with the even parity case, we can apply the anti-involution $ \vs $ in Lemma~\ref{lem:vs} to the formulas in Theorem~\ref{thm:iDP:ev} to obtain the following $F\kk\Y$-expansion formulas:
	
	\begin{prop}  \label{iDP:odd:FkY}
		For $m\ge 0$, we have
		\begin{align*}
		\dvo{2m} &= \sum_{c=0}^m \sum_{a=0}^{2m-2c} (-1)^c q^{-c+ a(2m-2c-a)}
		F^{(a)}  \LR{\kk;1+m-c}{c}  \Y^{(2m-2c-a)}, 
		\\
		\dvo{2m+1} &=  \sum_{c=0}^{m} \sum_{a=0}^{2m+1-2c} (-1)^c q^{c+ a(2m+1-2c-a)}
		F^{(a)}  \LR{\kk;1+m-c}{c}  \Y^{(2m+1-2c-a)}. 
		\end{align*}
	\end{prop}
	
	\begin{proof}
		This time $ \vs $ fixes $ F, \Y, J, K^{-1} $ and sends
		\[
		\dvo{n} \mapsto \dvo{n},\quad 
		\LR{\kk;a}{n}\mapsto (-1)^n q^{2n(n-1)}  \LR{\kk;2-a-n}{n}, \quad \forall a \in \Z,\; n\in \N. 
		\]
		The rest of the calculation is very similar to the even case above, and we obtain as before formulas that are formally the same as the non-super case, though there are factors of $ \pi $ and $ J $ contained in $ \LR{\kk;a+1}{n} $.
	\end{proof}

	For $ \lambda \in \Z $, recall from \ref{kqbinom2} that we have		\begin{equation}
	\LR{\kk;a}{n} \one_{2\la-1} = q^{2n(a-\la)} \qbinom{a-\la-1+n}{n}_{q^2} \one_{2\la-1} \in \UAdot_\od. 
	\end{equation}
	
	Hence, by a similar argument to the even parity case, we have the following reformulation of Theorem~\ref{thm:iDP:odd} (the extra factor of $ \pi^a $ comes from Lemma~\ref{dpY}):
	
	\begin{prop}  \label{iDP:odd:dot}
		For $m\ge 0$ and $\la \in \Z$, we have
		\begin{align*}
		\dvo{2m}  \one_{2\la-1}&= \sum_{c=0}^m \sum_{a=0}^{2m-2c} \pi^a (\pi q)^{ 2(a+c)(m-a-\la) -2ac+a -\binom{2c}{2} }  
		\qbinom{m-c-a-\la}{c}_{q^2} E^{(a)}  F^{(2m-2c-a)} \one_{2\la-1}, 
		\\
		\dvo{2m+1}  \one_{2\la-1}&=  \sum_{c=0}^{m} \sum_{a=0}^{2m+1-2c} 
		\\
		& \qquad \pi^a (\pi q)^{2(a+c)(m-a-\la)-2ac+2a  -\binom{2c}{2}} \qbinom{m-c-a-\la+1}{c}_{q^2} E^{(a)}  F^{(2m+1-2c-a)}  \one_{2\la-1}. 
		\end{align*}
		In particular, we have $\dvo{n}\one_{2\la-1} \in \UAdot_\od$, for all $n\in \N$. 
	\end{prop}
	
	\section{A Serre presentation of $\UUi$ and a $(q,\pi)$-binomial identity}
	\label{section:main}
	Let $ \UUi = \UUi_\vs $ be an $\imath{}$quantum group with parameter $ \vs $, for a given root datum $ (Y,X,\ang{\cdot,\cdot},\ldots )$.
	
	\begin{definition}
		For  $i\in I$ with $\tau i\neq i$, imitating Lusztig's divided powers,  we define the {\em divided power} of $B_i$ to be
		\begin{align}
		\label{eq:iDP1}
		B_i^{(m)}:=B_i^{m}/[m]_{i}^!, \quad \forall m\ge 0, \qquad \text{when } i \neq \tau i.
		\end{align}
		
		For $ i \in I $ with $ \tau i = i $, the {\em $\imath^\pi$-divided powers} are defined to be
		\begin{eqnarray}
		&& B_{i,\odd}^{(m)}=\frac{1}{[m]_{i}^!}\left\{ \begin{array}{ccccc} B_i\prod_{j=1}^k (B_i^2-\vs_i q_i [2j-1]_{i}^2 \tJ_i ) & \text{if }m=2k+1,\\
		\prod_{j=1}^k (B_i^2- \vs_i q_i [2j-1]_{i}^2 \tJ_i ) &\text{if }m=2k; \end{array}\right.
		\label{eq:piDPodd}\\
		&& B_{i,\even}^{(m)}= \frac{1}{[m]_{i}^!}\left\{ \begin{array}{ccccc} B_i\prod_{j=1}^k (B_i^2- \vs_i \pi_i q_i  [2j]_{i}^2 \tJ_i ) & \text{if }m=2k+1,\\
		\prod_{j=1}^{k} (B_i^2- \vs_i \pi_i q_i  [2j-2]_{i}^2 \tJ_i ) &\text{if }m=2k. \end{array}\right.
		\label{eq:piDPev}
		\end{eqnarray}
	\end{definition}
	
	When we specialize $ \pi_i = 1 $ and $ \tJ_i = 1 $, we obtain the $\imath$-divided powers in \cite{CLW18} from the formulas above. In the case when the parameter $\vs_i=q_i^{-1}$, this is the rank one case described in \S3, and all formulas and results there hold for $ B_{i,\ov{p}}^{(n)} $. In \ref{subsec:rescaling}, we obtain $ \UUi $ with general parameters $\vs_i$ from a special case by a rescaling isomorphism.

	\subsection{A Serre presentation of $\UUi$}\label{subsection:iserre presentation}
	
	Denote
	\[
	(a;x)_0=1, \qquad (a;x)_n =(1-a)(1-ax)  \cdots (1-ax^{n-1}), \quad \forall n\ge 1.
	\]
	For $ \UUi $ in the quantum covering setting, we have a {\em Serre presentation} result that parallels the main result in \cite{CLW18}, Theorem~3.1: Fix $\ov{p}_i\in \Z_2$ for each $i\in I$.

	\begin{thm}\label{thm:Serre}
		The $ \K(q)^\pi $-algebra $\UUi$ has a presentation with generators $B_i$, $\tJ_i$ $(i\in I)$, $K_\mu$ $(\mu\in Y^\imath)$ and the relations \eqref{relation1}--\eqref{eq:piSerre} below: for $\mu,\mu'\in Y^{\imath}$ and $i\neq j \in I$, 
		\begin{align}
		\tJ_i&\text{ is central},
		\label{relation1}
		\\	
		K_{\mu}K_{-\mu} &=1, 
		\quad
		K_\mu K_{\mu'}=K_{\mu+\mu'},
		\\
		K_\mu B_i-q_i^{-\langle \mu,\alpha_i\rangle} B_iK_\mu &=0, \quad
		\\
		[B_i,B_{j}] =&0, \quad \text{ if }a_{ij} =0 \text{ and }\tau i\neq j,
		\\
		\sum_{n=0}^{1-a_{ij}} (-1)^n \pi_i^{n p(j) + {n \choose 2}} B_i^{(n)}&B_jB_i^{(1-a_{ij}-n)} =0, \quad \text{ if } j \neq \tau i\neq i,
		\\
		\label{eq:piqSerre}
		\sum_{n=0}^{1-a_{i,\tau i}} (-1)^n {\pi_i^{n + {n \choose 2} }} B_i^{(n)}B_{\tau i}&B_i^{(1-a_{i,\tau i}-n)} =\frac{1}{{\pi_i} q_i-q_i^{-1}} \\
		\cdot \left(q_i^{a_{i,\tau i}} ({\pi_i} q_i^{-2};{\pi_i} q_i^{-2})_{-a_{i,\tau i}} B_i^{(-a_{i,\tau i})} \tJ_i \tK_i \tK_{\tau i}^{-1} 
		\right.& \left. 
		-({\pi_i} q_i^{2}; {\pi_i} q_i^{2})_{-a_{i,\tau i}}B_i^{(-a_{i,\tau i})} \tJ_{\tau i} \tK_{\tau i} \tK_{i}^{-1} \right), \text{ if } \tau i \neq i,
		\notag
		\\
		\sum_{n=0}^{1-a_{ij}} (-1)^n  {\pi_i^{n + {n \choose 2} }} B_{i,\overline{a_{ij}}+\overline{p_i}}^{(n)}&B_{j} B_{i,\overline{p}_i}^{(1-a_{ij}-n)}  =0, \quad \text{ if } \tau i = i \neq j.
		\label{eq:piSerre}
		\end{align}
	\end{thm}
	
	A proof of Theorem~\ref{thm:Serre} will be given in \S\ref{subsec:proofSerre}; first we will show that \eqref{eq:piqSerre} and \eqref{eq:piSerre} holds in $ \UUi $, in subsections \S\ref{subsec:qSerre} and \S\ref{subsec:Serre=T} respectively).
	
	Recall that a quasi-split $\imath${}quantum group $\UUi$ is split if $\tau=\id$. For split $\UUi$, its Serre presentation takes an particularly simple form, which we display here:
	\begin{thm}
		\label{thm:split}
		Fix $\ov{p}_i \in \Z_2$, for each $i\in I$. Then the split $\imath${}quantum group $\UUi$ has a Serre presentation with generators $B_i$ $(i\in I)$
		and relations
		\begin{align*}
		\sum_{n=0}^{1-a_{ij}} (-1)^n \pi_i^{n + {n \choose 2}} B_{i,\overline{a_{ij}}+\overline{p_i}}^{(n)}B_j B_{i, \ov{p}_i}^{(1-a_{ij}-n)}=0.
		\end{align*}
		Moreover, $\UUi$ admits a $\K(q)$-algebra anti-involution $\sigma$ which sends $B_i\mapsto B_i$ for all $i$.
	\end{thm}
	
	\begin{proof}
		Follows from Theorem \ref{thm:Serre} by noting that $Y^\bi=\emptyset$ and $\tau i=i$ for all $i\in I$.
	\end{proof}

	\subsection{Serre relation when $ \tau i \neq i $ }
	\label{subsec:qSerre}
	In this section we will show that \eqref{eq:piqSerre} holds, following \cite[Section~3.5]{BK15}. Recall the projections $ P_\lambda $ and $ \pi_{0,0} $ defined above, which are also in \cite{BK15}.
		
	\begin{prop}
		\label{prop:qSerre}
		If  $ \tau i \neq i $, the following relation holds in $ \UUi_{\vs} $:
		\begin{align*}
		\sum_{n=0}^{1-a_{i,\tau i}} (-1)^n \blue{\pi_i^{n + {n \choose 2} }} B_i^{(n)}B_{\tau i}B_i^{(1-a_{i,\tau i}-n)} =\frac{1}{\blue{\pi_i} q_i-q_i^{-1}} & \\
		\cdot \left(q_i^{a_{i,\tau i}} (\blue{\pi_i} q_i^{-2};\blue{\pi_i} q_i^{-2})_{-a_{i,\tau i}} B_i^{(-a_{i,\tau i})} \tJ_i \tK_i \tK_{\tau i}^{-1} 
		\right.& \left. 
		-(\blue{\pi_i} q_i^{2}; \blue{\pi_i} q_i^{2})_{-a_{i,\tau i}}B_i^{(-a_{i,\tau i})} \tJ_{\tau i} \tK_{\tau i} \tK_{i}^{-1} \right).
		\end{align*}
	\end{prop}
	
	\begin{proof}
		Recall now that $ i $ and $ j = \tau(i) \neq i $ must have the same parity, and if both $i$ and $j$ are even roots there is nothing to prove. Thus, we may assume that $i$ and $j$ are odd roots, and so by the bar-consistency condition $ m = 1 - a_{ij} $ is odd. Also set $ \lambda_{ij} = m \cdot i + j $ and with the notation above set $ Q_{-\lambda_{ij}} = \id \otimes ( P_{-\lambda_{ij}} \circ \pi_{0,0} ) $ as the vector space endomorphism of $ \UU \otimes \UU $. 
		
		By a construction parallel to \cite[(7.8)]{Ko14}, for $ Y = F_{ij}(B_i,B_j) $ we have the relation
		\begin{equation}
		C_{ij}(\mathbf{c}) = -(\id\otimes \varepsilon) \circ Q_{-\lambda_{ij}} (\Delta(Y) - Y \otimes K_{-\lambda_{ij}}).	
		\end{equation}
		
		
		Just as in \emph{loc. cit.}, we can compute $ \Delta(Y) $ from the formulas
		\begin{align*}
		\Delta(B_i) &= B_i \otimes K_i^{-1} + 1 \otimes F_i + \vs_i Z_i \otimes E_j K_i^{-1} \\
		\Delta(B_j) &= B_j \otimes K_j^{-1} + 1 \otimes F_j + \vs_j Z_j \otimes E_i K_j^{-1}
		\end{align*}
		where $ Z_k = J_{\tau(k)} K_{\tau(k)} K_k^{-1} $ for $ k = i,j $, 
		%
		and so we have that 
		
		\begin{equation}
		\label{Qlambda}
		Q_{-\lambda_{ij}} (\Delta(Y) - Y \otimes K_{\lambda_{ij}}) = (a_j B_i^{m-1} \vs_j Z_j + a_i B_i^{m-1} \vs_i Z_i) \otimes K_{-\lambda_{ij}}
		\end{equation}
		
		where $ a_i $ and $ a_j $ can be determined explicitly using the commutation relations
		$$
		Z_j B_i = q_i^{-(m+1)} B_i Z_j, \qquad Z_i B_i = q_i^{m+1} B_i Z_i.
		$$
		
		For instance,
		\begin{align*}
		a_j & B_i^{m-1}  \vs_j Z_j \otimes K_{-\lambda_{ij}} = Q_{-\lambda_{ij}} \biggl( \sum_{k = 0}^m (-1)^k \blue{\pi_i^{{k \choose 2} + k}} \qbinom{m}{k}_{i} \\
		& \quad \cdot \sum_{l = 0}^{m - k - 1} (B_i^l \otimes K_i^{-l}) (1 \otimes F_i) (B_i^{m - 1 - k - l} \otimes K_i^{-(m - 1 - k - l)}) (\vs_j Z_j \otimes E_i K_j^{-1}) (B_i^k \otimes K_i^{-k}) \biggr) \\
		& = \sum_{k = 0}^m { (-1)^k \blue{\pi_i^{{k \choose 2} + k} \pi_i } \over (\pi_i q_i - q_i^{-1}) } \qbinom{m}{k}_{i} \sum_{l = 0}^{m - k - 1} \blue{\pi_i^{m - 1 - l} \cdot \pi_i^k } q_i^{-(m+1)k - 2(m - k - l - 1)} B_i^{m-1} \vs_j Z_j \otimes K_{-\lambda_{ij}},
		\end{align*}
		where the extra factors of $ \blue{\pi_i} $ come from multiplying out $ 1 \otimes F_i $ and $ B_i^{m - 1 - k - l} \otimes K_i^{m - 1 - k -l} $ and $ B_i^k \otimes K_i^k $, and $ \vs_j Z_j \otimes E_i K_j^{-1} $ and $ B_i^k \otimes K_i^k $ respectively since multiplication in $ \UU \otimes \UU $ is defined according to the rule $ (a \otimes b)(c \otimes d) = \pi^{p(b)p(c)} ac \otimes bd $. The additional factor of $ \pi_i $ comes from the following: 
		\begin{align*}
		Q_{-\lambda_{ij}} ( K_i^{-(m - k -1)} F_i E_i K_j^{-1} K_i^{-k} ) &= Q_{-\lambda_{ij}} ( K_i^{-(m - k -1)} \bigg( \pi_i E_i F_i - \pi_i {J_i K_i - K_i^{-1} \over \pi_i q_i - q_i^{-1}} \bigg) K_j^{-1} K_i^{-k} ) \\
		&= { \pi_i \over \pi_i q_i - q_i^{-1} } K_i^{-m} K_j^{-1}.
		\end{align*} 
		\\
		Note that $ m - 1 = - a_{ij} $ is always even (by bar-consistency), and so $ \pi_i^{m-1} = 1 $. Thus,
		\begin{align*}
		a_j &= \sum_{k = 0}^m { (-1)^k \blue{\pi_i^{{k \choose 2}}} \over (\pi_i q_i - q_i^{-1}) } \qbinom{m}{k}_{i} \sum_{l = 0}^{m - k - 1} q_i^{ -(m-1)k - 2(m-1)} \pi_i^l q_i^{2l} \\
		&= \sum_{k = 0}^m { (-1)^k \blue{\pi_i^{{k \choose 2}}} \over (\pi_i q_i - q_i^{-1}) } \qbinom{m}{k}_{i} q_i^{ -(m-1)k - 2(m-1)} {  (\pi_i q_i^2)^{m - k} - 1 \over \pi_i q_i^2 - 1 }.
		\end{align*}  
		\\
		This time, we may use \cite[(1.12)]{CHW13}, which after applying the bar involution yields
		\begin{equation}
		\sum_{k = 0}^m \pi_i^{{k \choose 2}} q_i^{-k(m - 1)} \qbinom{m}{k}_{i} z^k = \prod_{j = 0}^{m - 1} (1 + (\pi_i q_i^{-2})^j z);
		\end{equation} 
		in particular, 
		$$
		\sum_{k = 0}^m \pi_i^{{k \choose 2}} q_i^{-k(m - 1)} \qbinom{m}{k}_{i} (-1)^k = 0;
		$$ 
		and
		$$
		\qquad \sum_{k = 0}^m \pi_i^{{k \choose 2}} q_i^{-k(m - 1)} \qbinom{m}{k}_{i} (-\pi_i q_i^{-2})^k = \prod_{j = 0}^{m - 1} (1 - (\pi_i q_i^{-2})^{j + 1} ) = (\pi_i q_i^{-2}; \pi_i q_i^{-2})_m,
		$$
		(Recall that $ (x;x)_m := \prod_{j = 1}^m (1 - x^j) $) and so (remembering that $ \pi_i^m = \pi_i $ since $m$ is odd) we have
		\begin{equation}
		a_j = { \pi_i q_i^{-2(m - 1)} (\pi_i q_i^2)^{m} \over q_i (\pi_i q_i - q_i^{-1})^2 } (\pi_i q_i^{-2}; \pi_i q_i^{-2})_m = {q_i \over (\pi_i q_i - q_i^{-1})^2 } (\pi_i q_i^{-2}; \pi_i q_i^{-2})_m.
		\end{equation}
		\\
		Similarly, for $ a_i $ we have additional factors of $ \blue{\pi_i^{{k \choose 2} + k}} $ from the super-Serre relations and $ \blue{\pi_i^l} $ from the tensor product multiplication:
		\begin{align*}
		a_i &= { \pi_i \over \pi_i q_i - q_i^{-1}} \sum_{k = 0}^m (-1)^k \blue{\pi_i^{{k \choose 2} + k}} \qbinom{m}{k}_{i} \sum_{l = 0}^{k - 1} q_i^{(k-1)(m+1)} \blue{\pi_i^l} q_i^{-2l} \\
		&= { \pi_i \over \pi_i q_i - q_i^{-1}} \sum_{k = 0}^m (-1)^k \pi_i^{{k \choose 2} + k} \qbinom{m}{k}_{i} q_i^{(k-1)(m+1)} {1 - (\pi_i q_i^{-2})^k \over 1 - \pi_i q_i^{-2}} \\
		&= { \pi_i (\pi_i q_i) \over (\pi_i q_i - q_i^{-1})^2 } q_i^{-(m+1)} \sum_{k = 0}^m (-1)^k \pi_i^{{k \choose 2}} \pi_i^k q_i^{k(m+1)} \qbinom{m}{k}_{i} (1 - (\pi_i q_i^{-2})^k) \\
		&= { q_i \over (\pi_i q_i - q_i^{-1})^2 } q_i^{-(m+1)} \sum_{k = 0}^m (-1)^k \pi_i^{{k \choose 2}} q_i^{k(m-1)} \qbinom{m}{k}_{i} ( (\pi_i q_i^2)^k - 1) \\
		&= { q_i^{-m} \over (\pi_i q_i - q_i^{-1})^2 }  \bigg( (\pi_i q_i^{2}; \pi_i q_i^{2})_m - 0 \bigg) = { q_i^{-m} \over (\pi_i q_i - q_i^{-1})^2 } (\pi_i q_i^{2}; \pi_i q_i^{2})_m,
		\end{align*}
		this time using \cite[(1.12)]{CHW13} directly (without the need for applying the bar involution).
		\\
		\\
		Putting this together with \ref{Qlambda} and applying $ - \id \otimes \varepsilon $, we obtain
		\begin{equation}
		\label{Citaui}
		C_{ij}(\mathbf{c}) = {-1 \over (\pi_i q_i - q_i^{-1})^2} ( q_i^{-m} (\pi_i q_i^{2}; \pi_i q_i^{2})_m B_i^{m - 1} \vs_i Z_i + q_i (\pi_i q_i^{-2}; \pi_i q_i^{-2})_m  B_i^{m - 1} \vs_j Z_j ).	
		\end{equation}
		Dividing through by $ [m]^!_{i} $ and simplifying yields the divided powers version presented in Theorem~\ref{thm:Serre}.
	\end{proof}

	\subsection{Change of parameters}
	\label{subsec:rescaling}
	
	In \cite{CLW18} (also cf. \cite[Theorem 7.1]{Ko14}), a change-of-parameters isomorphism is used to give a presentation of the $\imath${}quantum group $\UUi_{\vs,\kappa}$. In particular, it is shown that the $\K(q)$-algebra $\UUi_{\vs,\kappa}$ (up to some field extension) is isomorphic to $\UUi_{\vs^\diamond,{\bf0}}$ for some distinguished parameters $\vs^\diamond$, i.e., $\vs^\diamond=q_i^{-1}$ for all $i\in I$ such that $\tau i=i$ (cf. \cite{Le02}, \cite[Proposition~ 9.2]{Ko14}). The same argument carries over to the quantum covering setting: 
	
	For given parameters $\vs$ satisfying \eqref{bar1}--\eqref{bar3}, let $\vs^\diamond$ be the associated distinguished parameters such that $\vs_i^\diamond=\vs_i$ if $\tau i\neq i$, and
	\begin{equation}
	\label{eq:par0}
	\vs^\diamond_i=q_i^{-1}, \text{ if } \tau i=i.
	\end{equation}
	Let $\UUi_{\vs^\diamond}$ be the $\imath${}quantum covering group with the parameters $\vs^\diamond$ = for all $i\in I$.
	Let ${\F}= \K(q)(a_i\mid i\in I\text{ such that } \tau i=i)$ be a field extension of $\K(q)$, where
	\begin{align}
	a_i=\sqrt{q_i\vs_i}, \qquad \forall i\in I \text{ such that }\tau i=i.
	\end{align}
	Denote by $_{\F}\UUi_{\vs} =\F \otimes_{\K(q)} \UUi_{\vs}$ the $\F$-algebra obtained by a base change.
	
	\begin{prop}
		\label{prop:morphism}
		There exists an isomorphism of ${\mathbb F}$-algebras
		\begin{align*}
		\phi_\imath: {}_{\F}\UUi_{\vs^\diamond} & \longrightarrow {}_{\F}\UUi_{\vs},
		\\
		B_i \mapsto \left\{\begin{array}{ll} B_i, & \text{ if }\tau i\neq i, \\  a_i^{-1} B_i, & \text{ if } \tau i =i;\end{array} \right.
		\qquad K_\mu & \mapsto K_\mu, \quad (\forall i\in I, \mu\in Y^\bi),
		\end{align*}
	\end{prop}
	
	In particular, this allows us to use the formulas for $\imath^\pi $-divided powers in the previous section, free of unwieldy coefficients.

	\subsection{A $(q,\pi)$-binomial identity}
	
	We state and prove here a $ (q,\pi)$-binomial identity that will be crucial to the proof of Proposition~\ref{prop:iSerre:even} in the next section: for 
	\begin{equation}
	\label{eq:wul}
	w\in\Z, \quad u,\ell\in\Z_{\geq0}, \text{ with } u,\ell \text{ not both } 0,
	\end{equation}
	we define
	\begin{align}\label{eq:Tqpi}
	T& (w,u,\ell)_{q,\pi}  \\
	&= \sum_{\substack{c,e,r\geq0 \\ c+e+r=u}}
	\sum^{\ell}_{\substack{t=0 \\ 2\mid(t+w-r) }} 
	\notag\\
	&\quad \blue{\pi}^{lt + r + e + {t \choose 2}}
	(\blue{\pi} q)^{-t(\ell+u-1)+(\ell+u)(c-e)}
	\qbinom{\ell}{t} \qbinom{w+t-\ell}{r} \qbinom{u-1+\frac{w+t-r}{2}}{c}_{q^2}\qbinom{\frac{w+t-r}{2}-\ell}{e}_{q^2}
	\notag \\
	&-\sum_{\substack{c,e,r\geq0 \\ c+e+r=u}}
	\sum^{\ell}_{\substack{t=0 \\ 2\nmid(t+w-r) }} 
	\notag \\ \notag
	&\quad \blue{\pi}^{lt + c + {t \choose 2}}
	(\blue{\pi} q)^{-t(\ell+u-1)+(\ell+u-1)(c-e)}
	\qbinom{\ell}{t} \qbinom{w+t-\ell}{r} \qbinom{u+\frac{w+t-r-1}{2}}{c}_{q^2}\qbinom{\frac{w+t-r-1}{2}-\ell}{e}_{q^2}.
	\end{align}
	
	When we specialize at $ \pi = 1 $, we have $ T(w,u,\ell)_{q,1} = T(w,u,\ell) $ as defined in \cite[(3.18)]{CLW18}.
	
	\begin{prop}[\cite{CLW18}, Theorem~3.6]
		\label{prop:T=0}
		The identity $T(w,u,\ell)=0$ holds, for all integers $w, u, \ell$ as in \eqref{eq:wul}.
	\end{prop}
	
	As pointed out in \cite{CLW18}, a direct proof of this proposition proved challenging. Instead, the authors approached this by first introducing a more general $q$-binomial identity in several more parameters. This general identity specialized to the one above and satisfied certain recurrence relations, thus completing the proof with an inductive argument (details in \S5 of \cite{CLW18}). Fortunately for us, we can sidestep the complicated process above for the analogous result here in our setting by making a deft substitution and leveraging the earlier result:
	\begin{prop}
		\label{prop:Tpi=0}
		The identity $T(w,u,\ell)_{q,\pi}=0$ holds, for all integers $w, u, \ell$ as in \eqref{eq:wul}.
	\end{prop}

	
	\begin{proof}
		By a substitution of $ q \mapsto \sqrt{\pi} q $ in $ T(w,u,l) $, we obtain
		$$
		T(w,u,l) |_{q \mapsto \sqrt{\pi} q} = \sqrt{\pi}^{u^2 - lu - uw} T(w,u,\ell)_{q,\pi},
		$$
		and so the result follows from Proposition~\ref{prop:T=0}.
	\end{proof}

	\subsection{Proof of the $\imath^\pi$-Serre relations}
	\label{subsec:Serre=T}
	
	This section is devoted to a proof of the following theorem:
	
	\begin{thm}\label{thm:Serre=T}
		The $\imath^\pi$-Serre relations \eqref{eq:piSerre},
		$$
		\sum_{n=0}^{1-a_{ij}} (-1)^n  {\pi_i^{n + {n \choose 2} }} B_{i,\overline{a_{ij}}+\overline{p_i}}^{(n)} B_{j} B_{i,\overline{p}_i}^{(1-a_{ij}-n)}  =0, \quad \text{ if } \tau i = i \neq j.
		$$
		hold in the $\imath${}quantum covering group $\UUi$.
	\end{thm}
	
	The general strategy will rely on applying a few reductions to reduce \eqref{eq:piSerre} to the $(q,\pi)$-binomial above, which vanishes as we saw in Proposition~\ref{prop:Tpi=0}. Using the isomorphism $\phi$ in Proposition~\ref{prop:morphism}, the $\imath${}Serre relations for $\UUi_{q_i^{-1}}$ is transformed into the $\imath${}Serre relations \eqref{eq:piSerre} for $\UUi_{\vs}$ with general parameters. Hence just as in \cite{CLW18}, we will work with the $\imath${}quantum groups with distinguished parameters, $\UUi=\U_{q_i^{-1}}$, as a first reduction of the $\imath${}Serre relations. A subsequent `reduction by equivalence' as in \S4.1 of \cite{CLW18} can be applied, further reducing \eqref{eq:piSerre} to 
	\begin{equation}
	\sum_{n=0}^{1-a_{ij}} (-1)^n  B_{i,\ov{a_{ij}}+\ov{p}}^{(n)}F_jB_{i,\ov{p}}^{(1-a_{ij}-n)} =0
	\label{eqn:general F}
	\end{equation}
	for each $ \ov{p} \in \Z_2 $, where $ i \in \I $ such that $ \tau i = i $, $ j \neq i $.
	Now fix $ i = 1 $ and $ j = 2 $. Note that when $ p(1) $ is even, 
	there are no additional formulas to prove since $ \pi_1 = 1 $. Thus, we may assume that $ p(1) $ is odd, and so due to the bar-consistency condition (\cite[1.1(d)]{CHW13}) we must have $ a_{12} \in -2\N $. Hence, it is sufficient to prove that:
	%
	\begin{prop}
		\label{prop:iSerre:even}
		Suppose that $ a_{12} = -2m \in -2\N $. Then,
		\begin{align}
		\sum_{n=0}^{2m+1} (-1)^n \pi_1^{n p(2) + {n \choose 2}} B_{1,\even}^{(n)} F_2 B_{1,\even}^{(2m + 1 - n)} &= 0 ; \label{eq:piSerre:even} \\
		\sum_{n=0}^{2m+1} (-1)^n \pi_1^{n p(2) + {n \choose 2}} B_{1,\odd}^{(n)} F_2 B_{1,\odd}^{(2m + 1 - n)} &= 0. \label{eq:piSerre:odd}
		\end{align}
	\end{prop}

	\begin{proof}	
		Just as in \cite[\S4]{CLW18}, we will show that \eqref{eq:piSerre:even} holds by showing that
		\begin{equation}
		\label{eq:piSerremod:even}
		\sum_{n=0}^{2m+1} (-1)^n \pi_1^{n p(2) + {n \choose 2}} B_{1,\even}^{(n)} F_2 B_{1,\even}^{(2m + 1 - n)} \onestar_{2\lambda} = 0.
		\end{equation}
		for all $ \lambda $, using Remark~\ref{rem:u=0}. 
		
		Using Proposition~\ref{prop:iDPevdot} to expand $ B_{1,\even}^{(n)} $ and $ B_{1,\even}^{(2m + 1 - n)} $ and \eqref{eqn:commutate-idempotent3} to collect the factors of $ E_1$, we have (cf. \cite[(4.15)]{CLW18})
		
		\begin{align}\label{eq:evev}
		\small \sum_{n=0}^{2m+1} (-1)^n & \pi_1^{n p(2) + {n \choose 2}} B_{1,\even}^{(n)}F_2 B_{1,\even}^{(2m+1-n)}\onestar_{2\la}= \\ \notag
		&\small \sum_{n=0,2\mid n}^{2m}\sum_{c=0}^{m-\frac{n}{2}}\sum_{e=0}^{\frac{n}{2}} \sum_{a=0}^{2m+1-n-2c}\sum_{d=0}^{n-2e}\sum^{\min\{a,n-2e-d\}}_{r=0}
		\notag \\
		& \small \cdot
		\blue{\pi_1^{a + d + ap(2) + ad + {r \choose 2} + {n \choose 2} }} 
		(\blue{\pi_1} q_1)^{(a+c+d+e)(2m+1-n-2\la-2a-2c-2d-2e)+d} \notag \\
		&\small \cdot \qbinom{a+d-r}{d}_{q_1}\qbinom{2m+2-n-2\la-2e-d-3a-4c}{r}_{q_1} \qbinom{m-\frac{n}{2}-\la-c-a}{c}_{q_1^2}
		\notag\\
		&\small \cdot \qbinom{m+1-\frac{n}{2}-\la-e-d-2a-2c}{e}_{q_1^2}E_1^{(a+d-r)}F_1^{(n-2e-d-r)}F_2F_1^{(2m+1-n-2c-a)}\onestar_{2\la} \notag \\
		-&\small \sum_{n=1,2\nmid n}^{2m+1}\sum_{c=0}^{m+\frac{1-n}{2}}\sum_{e=0}^{\frac{n-1}{2}} \sum_{a=0}^{2m+1-n-2c}\sum_{d=0}^{n-2e}\sum^{\min\{a,n-2e-d\}}_{r=0} \notag \\
		& \small \cdot 
		\blue{\pi_1^{a + d + (a + 1)p(2) + a + ad + {r \choose 2} + {n \choose 2} }} 
		(\blue{\pi_1} q_1)^{(a+c+d+e)(2m+1-n-2\la-2a-2c-2d-2e)-a-2c} \notag \\
		&\small \cdot \qbinom{a+d-r}{d}_{q_1}\qbinom{2m+2-n-2\la-2e-d-3a-4c}{r}_{q_1} \qbinom{m+\frac{1-n}{2}-\la-c-a}{c}_{q_1^2}
		\notag\\
		&\small \cdot \qbinom{m+\frac{1-n}{2}-\la-e-d-2a-2c}{e}_{q_1^2}E_1^{(a+d-r)}F_1^{(n-2e-d-r)}F_2F_1^{(2m+1-n-2c-a)}\onestar_{2\la}. \notag
		\end{align}
		
		By the same series of substitutions as detailed in \cite{CLW18}, we may collect the $q$- and $q^2$-binomial factors and some of the $ \pi_1 $ factors into a sum $ S(y,u,\ell,\lambda)_\pi $ (the rest can be factored out) to obtain
		\begin{align}
		\label{eq:BFB=S}
		&\sum_{n=0}^{2m+1} (-1)^n \pi_1^{n p(2) + {n \choose 2}} B_{1,\even}^{(n)}F_2 B_{1,\even}^{(2m+1-n)}\onestar_{2\la} =
		\sum_{\substack{ \ell,y,u\geq0;u+\ell>0 \\ \ell+y+2u\leq 2m+1}} \\
		&\quad
		\blue{\pi_1^{(l + y)p(2) + l +  {y \choose 2}}}
		(\blue{\pi_1} q_1)^{(\ell+u)(2m+1-2\la-2\ell-3u-y)} S(y,u,\ell,\la)_\pi E_1^{(\ell)}F_1^{(y)}F_2F^{(2m+1-\ell-y-2u)}_1\onestar_{2\la},
		\notag
		\end{align}
		where $ S(y,u,\ell,\lambda)_\pi $ is a sum over $ n $ (with a difference when $2 | n$ and $ 2 \nmid n $ as above ) and over $ c,e,r \geq 0, c+e+r = u $ cf. \cite[4.16]{CLW18}.
		
		Then, using the new variables $ t := -u-y-e+c+n $ and $ w := 2m +2 - 2 \la -2l - 4u - y $ in \S4.4 of \cite{CLW18}, we have that $ S(y,u,\ell,\lambda)_\pi = T(w,u,\ell)_{q,\pi} $. Thus, the right-hand side vanishes by Theorem~\ref{prop:Tpi=0} and so \eqref{eq:piSerre:even} holds.
		Just as in \cite{CLW18}, a similar argument shows that \eqref{eq:piSerre:odd} holds.
	\end{proof}

	\subsection{Proof of Theorem \ref{thm:Serre}}
	\label{subsec:proofSerre}
	
%
%
%
	We have a generalization of \cite[Theorem~7.1]{Ko14} when $X$ (corresponding to black nodes) is empty; the main ingredients are the results in \S\ref{UUisize} above. Finally, the computation of the `Serre correction terms' $ C_{ij} $ is given by \eqref{eq:piqSerre} and \eqref{eq:piSerre}, whose validity we have shown via Proposition~\ref{prop:qSerre} and Theorem~\ref{thm:Serre=T}. 
	%
	\qed
	\subsection{Bar involution on $\UUi$}
	\label{subsec:bar} 
	
	Recall the three conditions \eqref{bar1}--\eqref{bar3} on $ \vs_i $ in Definition~\ref{def:UUi}. We may now conclude the existence of the bar involution for the quasi-split $\imath${}quantum group $\UUi:=\UUi_{\vs}$, granting that these conditions on $\vs_i$ are satisfied:
	
	\begin{prop}
		\label{prop:bar}
		Assume the parameters $\vs_i$, for $i\in I$, satisfy the conditions \eqref{bar1}--\eqref{bar3}, which we recall here:
		\begin{enumerate}
			\item[\eqref{bar1}] $\ov{\vs_iq_i} =\vs_iq_i$, if $\tau i=i$ and $a_{ij}\neq 0$ for some $j\in I\setminus\{i\}$;
			\item[\eqref{bar2}] $\ov{\vs_i} =\vs_i =\vs_{\tau i}$, if $\tau i\neq i$ and $a_{i,\tau i}=0$;
			\item[\eqref{bar3}] $\vs_{\tau i}= \blue{\pi_i} q_i^{-a_{i,\tau i}}\ov{\vs_i}$, if $\tau i\neq i$ and $a_{i,\tau i}\neq0$.
		\end{enumerate}
		Then there exists a $\K$-algebra automorphism $ ^{\ov{\,\,\,\,\,}}: \UUi\rightarrow \UUi$ (called a bar involution) such that
		\[
		\ov{q}=q^{-1}, \quad
		\ov{K_\mu}=K_{\mu}^{-1}, \quad
		\ov{B_i}=B_i,  \quad
		\forall \mu\in Y^\bi, i\in I.
		\]
	\end{prop}
	
	\begin{proof}
		Under the assumptions, the $\imath$-divided powers $B_i^{(n)}$ in \eqref{eq:iDP1} and $B_{i, {\ov{p}}}^{(n)}$, for $\ov{p} \in \Z_2$, in \eqref{eq:piDPodd}-\eqref{eq:piDPev} are clearly bar invariant. It follows by inspection that all the explicit defining relations for $\UUi$ in \eqref{relation1}-\eqref{eq:piSerre} are bar invariant. The extra factor of $ \pi_i $ in (c) comes from applying $ ^{\ov{\,\,\,\,\,}} $ to the right hand side of \eqref{eq:piqSerre}. 
	\end{proof}
	
	For the bar-involution defined above, we will construct a quasi $K$-matrix $ \Upsilon $ and develop a theory of canonical bases for $ \UUi $ in a follow up \cite{C19b} to this paper, cf. \cite{BW18b,BW18c}.

	

\end{document}